\documentclass[11pt]{amsart}
\usepackage[left=1.5in,top=1.15in,right=1.5in,bottom=1.15in]{geometry}
\usepackage{amssymb,amsmath,amsthm,amsfonts,amsopn,url,hyperref}
\usepackage[all]{xy}
\usepackage{enumitem}
\usepackage{amscd}
\theoremstyle{plain}
\newtheorem{thm}{Theorem}

\newtheorem{prop}[thm]{Proposition}
\newtheorem{lemma}[thm]{Lemma}
\newtheorem{corollary}[thm]{Corollary}
\theoremstyle{definition}
\newtheorem{defn}[thm]{Definition}
\newtheorem*{defn*}{Definition}

\theoremstyle{remark}
\newtheorem{remark}[thm]{Remark}

\newtheorem*{remark*}{Remark}
\newtheorem{claim}[equation]{Claim}
\newtheorem*{claim*}{Claim}
\usepackage[T1]{fontenc}
\usepackage{aurical}
\usepackage{carolmin}
\usepackage{uncial}
\usepackage{inslrmin}

\numberwithin{thm}{section}
\numberwithin{equation}{thm}

\newcommand{\field}[1]{\mathbb{#1}}
\newcommand{\N}{\field{N}}

\newcommand{\ideal}[1]{\mathfrak{#1}}
\newcommand{\m}{\ideal{m}}
\newcommand{\n}{\ideal{n}}
\newcommand{\p}{\ideal{p}}

\newcommand{\ia}{\ideal{a}}

\newcommand{\ffunc}[1]{\mathrm{#1}}
\newcommand{\func}[1]{\mathrm{#1} \,}

\newcommand{\Spec}{\func{Spec}}

\newcommand{\Ext}{\ffunc{Ext}}

\newcommand{\im}{\func{im}}

\newcommand{\arrow}[1]{\stackrel{#1}{\rightarrow}}

\newcommand{\ra}{\rightarrow}
\newcommand{\onto}{\twoheadrightarrow}

\renewcommand{\epsilon}{\varepsilon}

\DeclareMathOperator{\ann}{ann}
\DeclareMathOperator{\Hom}{Hom}

\DeclareMathOperator{\Max}{Max}

\newcommand{\tint}{tight\ interior}

\newcommand{\red}{\mathrm{red}}

\newcommand{\taub}{{\tau_b}}
\newcommand{\taufg}{{\tau_{\textnormal{fg}}}}
\newcommand{\cf}{{\itshape cf.} }
\newcommand{\ie}{{\itshape i.e.} }
\newcommand{\tensor}{\otimes}
\newcommand{\mydot}{{{\,\begin{picture}(1,1)(-1,-2)\circle*{2}\end{picture}\ }}}
\newcommand{\mydt}{{{\,\begin{picture}(1,1)(-1,-2)\circle*{2}\end{picture}}}}

\newcommand{\NPullback}[2]{\text{\Fontskrivan\slshape{F}}^{\hspace{2pt}#2}_{#1}}

\newcommand{\bc}{\mathfrak{c}}
\newcommand{\ba}{\mathfrak{a}}
\newcommand{\bp}{\mathfrak{p}}
\newcommand{\resp}{{\emph{resp.}\ }}
\newcommand{\ov}[1]{{\overline{{#1}}}}
\DeclareMathOperator{\Image}{Image}
\DeclareMathOperator{\Ann}{Ann}
\newcommand{\myH}{{\bf h}}

\DeclareFontFamily{OMS}{rsfs}{\skewchar\font'60}
\DeclareFontShape{OMS}{rsfs}{m}{n}{<-5>rsfs5 <5-7>rsfs7 <7->rsfs10 }{}
\DeclareSymbolFont{rsfs}{OMS}{rsfs}{m}{n}
\DeclareSymbolFontAlphabet{\scr}{rsfs}
\newcommand{\sC}{\scr{C}}
\newcommand{\sB}{\scr{B}}
\newcommand{\sD}{\scr{D}}
\newcommand{\noMinPrime}{\circ}

\usepackage{color}

\author{Neil Epstein}
\address{Department of Mathematical Sciences \\ George Mason University \\ Fairfax, VA  22030}
\email{nepstei2@gmu.edu}
 \thanks{The first author was partially supported by a grant from the German Research Foundation (DFG)}

\author{Karl Schwede}
\address{Department of Mathematics \\
The Pennsylvania State University \\
University Park, PA 16802 \\
USA}
\email{schwede@math.psu.edu}
\thanks{The second author was partially supported by NSF DMS \#{}1064485/0969145 and by an NSF postdoctoral fellowship \#{}0703505}

\subjclass[2010]{13A35, 13B22, 13B40, 14B05, 14F18}
\keywords{tight closure, test ideal, interior operation, persistence, Cartier algebra, normalization, conductor ideal}

\date{\today}
\title{A dual to tight closure theory}

\begin{document}

\begin{abstract}
We introduce an operation on modules over an $F$-finite ring of characteristic $p$.  We call this operation \emph{tight interior}.  While it exists more generally, in some cases this operation is equivalent to the Matlis dual of tight closure.  Moreover, the interior of the ring itself is simply the big test ideal.  We directly prove, without appeal to tight closure, results analogous to persistence, colon capturing, and working modulo minimal primes, and we begin to develop a theory dual to phantom homology.

Using our dual notion of persistence, we obtain new and interesting transformation rules for tight interior, and so in particular for the test ideal, which complement the main results of a recent paper of the second author and K.~Tucker.  Using our theory of phantom homology, we prove a vanishing theorem for maps of Ext.  We also compare our theory to M.~Blickle's notion of Cartier modules, and in the process, we prove new existence results for Blickle's test submodule.  Finally, we apply the theory we developed to the study of test ideals in non-normal rings, proving that the finitistic test ideal coincides with the big test ideal in some cases.
\end{abstract}
\maketitle

\section{Introduction}

Tight closure is an operation on modules over a commutative ring of characteristic $p > 0$.  Indeed, given any modules $N \subseteq M$ over a ring $R$, the tight closure of $N$ in $M$ is a submodule of $M$, $N^* \supseteq N$.  Tight closure has had many interesting applications, but has turned out to be a decidedly non-geometric operation since it does not commute with localization \cite{BrennerMonsky}.  In this paper we develop a dual theory to tight closure that does commute with localization.

Indeed, suppose that $M$ is any $R$-module.  We introduce a new operation, the \emph{tight interior} of $M$.  This is a submodule of $M$ which we denote by $M_{*R}$ or simply $M_*$ if the context is clear (see Section \ref{sec.Definitions} for the definition).  
In the case that $R$ is local, complete, and $F$-finite and $M$ is finitely generated, then the tight interior operation just corresponds the Matlis dual of $(M^{\vee}) / (0^*_{M^{\vee}})$, see Corollary \ref{cor.DualityWithTC}.  However, the theory seems well behaved in greater generality (although we still largely work in the $F$-finite case).  For example, the construction of $M_*$ always commutes with localization.  Furthermore, we show that many of the key properties of tight closure -- persistence, colon capturing, working modulo minimal primes and others -- have direct analogs for this operation (which we prove directly without appeal to tight closure theory), see for example Proposition \ref{pr:minprimes}, Theorem \ref{thm:copers}, and Theorem \ref{thm.CoColonCapturing}.  We state one variant of persistence in this setting below in order to illustrate our meaning.
\vskip 6pt
\noindent
{\bf{Theorem \ref{thm:copers}.}} [Co-persistence]{\it{
Let $j: R \ra S$ be a ring homomorphism such that both $R_\red$ and $S_\red$ are $F$-finite.  Let $M$ be an $R$-module.  Then the natural evaluation map $\epsilon: \Hom_R(S,M) \ra M$ restricts to the $(-)_*$ level.  In other words we have a map
\[
 \left( \Hom_R(S,M) \right)_{* S} \to M_{* R}
\]
induced by restriction of $\epsilon$.
}}
\vskip 6pt
Interestingly, for a reduced $F$-finite ring $R$, if we view $R$ as a module over itself, then $R_*$ is simply the big test ideal of $R$.  This was essentially pointed out in \cite{HaraTakagiOnAGeneralizationOfTestIdeals} but Matlis-dual statements have a much longer history, see for example \cite[Proposition 8.23(c)]{HochsterHunekeTC1}, \cite[Proposition 4.4]{SmithTestIdeals}, \cite{SmithTightClosureParameter}, \cite{LyubeznikSmithStrongWeakFregularityEquivalentforGraded}, \cite{LyubeznikSmithCommutationOfTestIdealWithLocalization}.  In fact, this body of work motivates our definition in general.  Indeed, many of the properties of tight interior mentioned above (working modulo minimal primes, persistence, colon capturing etc.) lead to interesting and useful statements simply for big test ideals.

One of the most interesting notions to come out of tight closure theory is that of phantom homology and phantom resolutions.  In Section \ref{sec.Cophantom}, we develop a dual theory.  One of the most celebrated results in tight closure theory is the vanishing theorem for maps of Tor.  Using our dual theory of cophantom resolutions and co-persistence, we directly obtain a vanishing theorem for maps of Ext, stated below.
\vskip 6pt
\noindent
{\bf{Theorem \ref{thm:VanishingForExt}.}} [Vanishing theorem for maps of Ext] {\it{
Consider a sequence of ring homomorphisms $A \hookrightarrow R \ra T$ such that $T$ is $F$-finite and regular (or simply strongly $F$-regular), $R$ is a module-finite torsion-free extension of $A$, $A$ is a domain, and both $A$ and $R$ are $F$-finite.  Let $M$ be a finite $A$-module of finite injective dimension.  Then for all $i\geq 1$, the natural maps $\Ext^i_A(T, M) \ra \Ext^i_A(R, M)$ are zero.
}}
\vskip 6pt
Another motivation for this work is to develop connections with recent work of M.~Blickle.  In \cite{BlickleTestIdealsViaAlgebras}, \cf \cite{BlickleBoeckleCartierModulesFiniteness}, a theory of test submodules was developed.   Suppose that $R$ is an $F$-finite ring, $M$ is a finitely generated $R$-module, and finally fix a graded ring $\sD$ of maps $\phi : {}^e M \to M$ (for various $e > 0$) with multiplication via composition \cf \cite{LyubeznikSmithCommutationOfTestIdealWithLocalization,SchwedeTestIdealsInNonQGor}.  In this case, Blickle associated a submodule $\tau(M, \sD) \subseteq M$ which he called the test submodule of $M$ with respect to $\sD$ (although existence of this submodule is an open question in general), see Section \ref{sec.CartierModules}.  However, for a general module $M$, if we pick the canonical choice of graded ring $\sC_M$ (namely, $\sC_M$ is made up off all possible maps), then we obtain the following theorem which also proves existence of $\tau(M, \sD)$ in a new case.
\vskip 6pt
\noindent
{\bf{Theorem \ref{thm.InteriorVsBlickle}.}} {\it{ Suppose that $R$ is an $F$-finite reduced ring and that $M$ is a finitely generated $R$-module whose support is equal to the support of $R$.  Then $M_* = \tau(M, \sC_M)$.  In particular, $\tau(M, \sC_M)$ exists.
}}
\vskip 6pt
\noindent Note that in the case that $M = R$, $\tau(R, \sC_R)$ merely coincides with the big test ideal $\tau_b(R)$ motivating Manuel Blickle's original definition.  In the case that $M = R$, the fact that $R_* = \tau(R, \sC_R)$ was essentially proven in \cite{LyubeznikSmithCommutationOfTestIdealWithLocalization} \cf \cite{SchwedeFAdjunction,SchwedeTestIdealsInNonQGor}.

Motivated by our observations on this interior operation, especially with regards to its behavior modulo minimal primes, we also study the behavior of the test ideal for non-normal rings.  In particular, we obtain the following theorem which can be viewed as a variant of \cite[Proposition 4.4]{SmithMultiplierTestIdeals}:
\vskip 6pt
\noindent
{\bf{Theorem \ref{thm.TauForNormalizationFRegular}.}} \cf [Theorem \ref{thm.ModMinimalPrimesIsFRegularThen}] {\it{
Suppose that $R$ is an $F$-finite reduced ring and that $R \subseteq R^{\textnormal{N}}$ is its normalization with conductor $\bc$.  If $R^{\textnormal{N}}$ is strongly $F$-regular, then
\[
\bc = \taufg(R) = \taub(R).
\]
}}
\vskip 0pt
\noindent Here $\bc$ is the conductor ideal and $\tau_{\textnormal{fg}}(R)$ is the finitistic (or classical) test ideal as originally defined in \cite{HochsterHunekeTC1}.

In order to prove this, we show several transformation rules for tight interior (and thus for big test ideals) under ring maps, these also rely on co-persistence mentioned above.  We also explore the behavior of both the big and finitistic test ideal under normalization in general, see Section \ref{sec:conductor}.  However, the transformation rule for \tint\ under finite maps should be of particular interest.  We state this result below.
\vskip 6pt
\noindent
{\bf{Corollary \ref{cor.TauTransformsUnderFiniteExtensions}.}} {\it{Suppose that $R$ is a $F$-finite reduced ring of characteristic $p > 0$ and $R \subseteq S$ is a finite extension with $S$ reduced.  Further suppose that $L$ is a finite $R$-module whose support agrees with $R$ and $M$ is a finite $S$-module whose support agrees with $S$.  Then:
\begin{equation*}
\label{eq.InteriorDescriptionForFiniteMaps}
L_* = \sum_{e \geq 0} \sum_{\phi} \phi( {}^e(M_*) )
\end{equation*}
where $\phi$ ranges over all elements of $\Hom_{R}( {}^e M, L)$.

In particular, if $L = R$ and $M = S$, then
\begin{equation*}
\label{eq.TauDescriptionForFiniteMaps}
\taub(R) = \sum_{e \geq 0} \sum_{\phi} \phi( {}^e\taub(S) )
\end{equation*}
where $\phi$ ranges over all elements of $\Hom_{R}( {}^e S, R)$.
}}
\vskip 6pt
This result should be viewed as complementary to several of the main results of \cite{SchwedeTuckerTestIdealFiniteMaps}.  In particular, this implies that the main result of \cite{SchwedeTuckerTestIdealFiniteMaps} is closely related to persistence in tight closure.

\begin{remark}
\label{rem.FormalDuality}
As mentioned in the introduction, many of the results of this paper can be viewed as a formal dual of the results of tight closure (even though they are applied to modules for which Matlis duality need not apply).  Indeed, a number of the theorems contained here-in use roughly the same proofs as in tight closure theory once we make the following identifications:
\begin{center}
\begin{tabular}{l|r}
notion & dual notion\\
\hline \hline
kernel & image\\
\hline
sum & intersection\\
\hline
tensor & (covariant) $\Hom$
\end{tabular}
\end{center}
\end{remark}
\vskip 6pt
\noindent {\it Acknowledgements:}  The authors would like to thank Manuel Blickle and Kevin Tucker for many valuable conversations.  They are also indebted to the referee, Manuel Blickle, Alberto Fernandez Boix, Karen Smith and Kevin Tucker for useful comments on a previous draft of the paper.

\section{Definitions and basic properties of \tint}
\label{sec.Definitions}
Let $R$ be a Noetherian ring of prime characteristic $p>0$, such that $R_\red$ is an $F$-finite Noetherian ring of prime characteristic $p>0$.  Use the usual conventions $q=p^e$ and $q_j = p^{e_j}$, $q' = p^{e'}$, etc.  For an $R$-module $M$ and integers $e\geq 0$, let ${}^eM$ denote the $R$-$R$ bimodule, with element set ${}^eM = \{^ex \mid x\in M\}$ formally the same as that of $M$, with the same additive structure, and with $R$-$R$ bimodule structure given by $r \cdot  ({}^ex) \cdot s = {}^e(r^q sx)$ for $r,s \in R$ and ${}^ex \in {}^eM$.


Let $M$ be an $R$-module.  For a power $q_0$ of $p$ and $c\in R^\circ$ (recall that $R^{\circ}$ is the elements of $R$ not contained in any minimal prime), let
\[
M_*[c,q_0] := \sum_{q\geq q_0} \im (\Hom_R({}^eR, M) \ra M),
\]
where the map in question sends a map $g$ to $g({}^ec)$.  Then we define the \emph{\tint} of $M$, $M_*$, via:
\[
M_* := \bigcap_{c\in R^o} \bigcap_{e_0 \geq 0} M_*[c, p^{e_0}].
\]
We will consider how the \tint\ changes as we vary the ring we are working over.  Therefore, if $M$ is both an $R$ and $S$-module, then we use $M_{*R}$ and $M_{*S}$ to denote the \tint\ of $R$ as an $R$-module and $S$-module respectively.

We start by first observing how this operation behaves with respect to module maps.

\begin{lemma}
\label{lem.ModuleMapsRespectInterior}
Let $f : N \to M$ be a map of $R$-modules.  Then $f({N_*}) \subseteq M_*$.
\end{lemma}
\begin{proof}
Suppose $z \in N_*$.  Thus for every $c \in R^{\noMinPrime}$, and ever $e_0 \geq 0$, there exists $e_1, \dots, e_n > e_0$ and $\phi_{i} : {}^{e_i} R \to N$, $e_i \geq e_0$ such that $z = \sum_{i = 1}^n \phi_i({}^{e_i} c)$.  But then $f(z) = \sum_{i = 1}^n f(\phi_i({}^{e_i} c))$.
\end{proof}

Of course, the above map is not generally surjective.
We now note that when computing interiors, we may `reduce to the reduced case':
\begin{prop}\label{pr:red}
Let $M$ be an $R$-module, and $\n$ the nilradical of $R$, so that $R_\red = R/\n$.  Then \[
M_{*R} = (0 :_M \n)_{*R_\red}.
\]
\end{prop}

\begin{proof}
Let $e_0$ be a fixed integer which is large enough that $\n^{[p^{e_0}]} = 0$.  Then for any $e\geq e_0$, we have \begin{equation}
\label{eq.HomsAreTheSame}
\begin{array}{rcl}
\Hom_R({}^eR, M) &= & \Hom_R( {R_\red} \otimes_{R_\red} {}^eR , M) \\
&\cong & \Hom_{R_\red}({}^eR, \Hom_R(R/\n, M)) \\
&\cong & \Hom_{R_\red}({}^eR, (0 :_M \n)).
\end{array}
\end{equation}
All these isomorphisms are fully canonical, and from the surjection $R \onto R_\red$, exactness of the ${}^e(-)$ functor, and left-exactness of the $\Hom$ functor, we have
a canonical injection $\Hom_{R_\red}({}^e(R_\red), (0 :_M\n)) \hookrightarrow \Hom_{R_\red}({}^eR, (0:_M\n))$ given by restriction.  Combine this with the displayed equation \eqref{eq.HomsAreTheSame} and tracing
what happens to elements, and we see that $M_{*R} \supseteq (0 :_M \n)_{*R_\red}$.

Conversely, for any $d \geq 0$, consider the map ${}^{d} (F^{e_0}) : {}^{d} R \to {}^{d + e_0} R$ which sends \mbox{${}^d r \mapsto {}^{d+e_0} (r^{p^{e_0}})$}.  This induces a map $\Psi : \Hom_R({}^{d+e_0} R, (0 :_M \n)) \to \Hom_R({}^d R, (0 :_M \n))$ with $\Psi(\phi)({}^d r) = \phi({}^{d + e_0}(r^{p^{e_0}}))$ and hence a map
\[
\Phi : \Hom_R({}^{d+e_0} R, (0 :_M \n)) \to \Hom_R({}^d R_{\red}, (0 :_M \n))
\]
since $\n^{[p^{e_0}]} = 0$.  Obviously $\Hom_R({}^d R_{\red}, (0 :_M \n)) = \Hom_{R_{\red}}({}^d R_{\red}, (0 :_M \n))$ and by \eqref{eq.HomsAreTheSame} we have $\Hom_R({}^{d+e_0} R, (0 :_M \n)) \cong \Hom_R({}^{d+e_0} R, M)$.  Combining these isomorphisms gives us
\[
\Phi' : \Hom_R({}^{d+e_0} R, M) \to \Hom_{R_{\red}}({}^d R_{\red}, (0 :_M \n)).
\]
Finally, for any $c \in R^{\circ}$, we have ${}^{d +e_0}(c^{p^{e_0}}) \in {}^{d+e_0} R^{\circ}$.  It follows from construction that $(\Phi'(\phi))({}^d c) = \phi({}^{d+e_0}(c^{p^{e_0}})) \in (0 :_M \n) \subseteq M$.  Thus for any $d \geq 0$
\[
M_{* R}[c^{p^{e_0}}, p^{e_0 + d}] \subseteq (0 :_M \n)_{* R_{\red}}[c, p^d].
\]
Summing over these terms completes the proof.
\end{proof}

An element $c\in R^\circ$ is called a \emph{$q_0$-weak co-test element} if $M_* = M_*[c, q_0]$ for all $R$-modules $M$.  We call $c$ a \emph{co-test element} if $M_* = M_*[c,1]$ for all $R$-modules $M$.   It is clear by definition that $M_* \subseteq M$ for all $R$-modules $M$.  To see that co-test elements exist, we show that in cases we care about, they coincide both with so-called ``big'' test elements \emph{and}, as a bonus, with the nontrivial elements of $R_*$:

\begin{prop}\label{pr:bigtestint}
Let $R$ be an $F$-finite reduced ring.  Then $\taub(R) = R_*$.
\end{prop}

\begin{proof}
This is a direct application of the equivalence ``(i) $\iff$ (ii)'' from  \cite[Lemma 2.1]{HaraTakagiOnAGeneralizationOfTestIdeals} with $\ia =R$, and $r_e=1=x_1^{(e)}$ for all $e$.
\end{proof}

\begin{remark}\label{rmk:locint}
Let $R$ be an $F$-finite reduced ring.  By \cite{LyubeznikSmithCommutationOfTestIdealWithLocalization} (also see \cite[Lemma 2.1]{HaraTakagiOnAGeneralizationOfTestIdeals}), we have $W^{-1} \taub(R) = \taub(W^{-1}R)$ for any multiplicative subset $W$ of $R$, and hence we have $W^{-1} (R_{*R}) = (W^{-1}R)_{*(W^{-1}R)}$.

By the same remark, if $R$ is local then $\taub(R) \otimes_R \hat{R} = \taub(\hat{R})$, and hence we have $\widehat{R_{*R}} = \hat{R}_{*\hat{R}}$.
\end{remark}

\begin{thm}\label{thm:cotestint}
Let $R$ be an $F$-finite reduced ring.  Then $R^\circ \cap R_*$ is precisely the set of all co-test elements of $R$.

Hence, for any $c\in R^\circ$, $c$ is a co-test element $\iff c\in R_* \iff c \in \taub(R)$.
\end{thm}

\begin{proof}
It is clear from the definitions that any co-test element $c$ of $R$ is in $R_*$.  To see this, simply note that $R_* = R_*[c,1]$ and consider the identity map on $R$.

Conversely let $c \in R^\circ \cap R_*$, and let $M$ be an arbitrary $R$-module.  Let $q'=p^{e'}$ be a power of $p$, let $g \in \Hom_R(R^{1/q'}, M)$, and let $z=g(c^{1/q'})$.  Let $d\in R^\circ$ and $e_0 \geq 0$ an integer.  Since $c\in R_*$, there is some $e_1\geq e_0$ and $R$-linear maps $\phi^{(e)}: R^{1/q} \ra R$ such that \[
c=\sum_{e=e_0}^{e_1} \phi^{(e)}(d^{1/q}).
\]
Then for each such $e$, $(\phi^{(e)})^{1/q'}$ is an $R^{1/q'}$-linear map (hence also an $R$-linear map) from $R^{1/qq'}$ to $R^{1/q'}$.  And we have \begin{align*}
z &= g(c^{1/q'}) = \sum_{e=e_0}^{e_1} g([\phi^{(e)}(d^{1/q})]^{1/q'})\\
&= \sum_{e=e_0}^{e_1} g((\phi^{(e)})^{1/q'}(d^{1/qq'})) = \sum_{e=e_0}^{e_1} (g \circ (\phi^{(e)})^{1/q'})(d^{1/qq'}) \\
&= \sum_{e=e_0+e'}^{e_1+e'} (g \circ (\phi^{(e-e')})^{1/q'})(d^{1/q}).
\end{align*}
But each $g \circ (\phi^{(e-e')})^{1/q'} \in \Hom_R(R^{1/q}, M)$.  Hence, $z\in M_*[d, q'q_0] \subseteq M_*[d, q_0]$.

Since every element of $M_*[c,1]$ is generated by elements like the $z$ given above, it follows that $M_*[c,1] \subseteq M_*$ (since $d$ and $q_0$ were chosen arbitrarily), whence $M_*=M_*[c,1]$.  Since $M$ was arbitrary, $c$ is a co-test element.

The last statement of the theorem follows by combining the first statement with Proposition~\ref{pr:bigtestint}.
\end{proof}

Hence, by Remark~\ref{rmk:locint}, ``completely stable'' co-test elements exist in the strong sense that if $c$ is a co-test element, so is $c/1 \in W^{-1}R$ for any multiplicative set $W$, and so is $c/1 \in \widehat{R_\p}$ for any $\p \in \Spec R$.

Let $c$ be a co-test element.  It is clear from the definition that for all $d\in R^\circ$, $dc$ is also a co-test element.


\begin{remark}
Since big test elements coincide with co-test elements at least in most of the cases of interest in this work, we often use the term big test element, instead of co-test element in order for the language contained in this paper to appear more familiar to experts.
\end{remark}

\begin{prop}\label{pr:idem}
Let $R$ be a ring such that $R_\red$ is $F$-finite.  Then for any $R$-module $M$, $(M_*)_* = M_*$.
\end{prop}

\begin{proof}
It is clear from the definition that $N_* \subseteq N$ for all $N$.  So we need only show that $M_* \subseteq (M_*)_*$.

Suppose first that $R$ is not reduced, and let $\n$ be the nilradical of $R$.  Then by Proposition~\ref{pr:red}, we have \[
(M_{*R})_{*R} = (0 :_{M_{*R}} \n)_{*R_\red} = (0 :_{(0 :_M \n)_{*R_\red}} \n)_{*R_\red} = ((0:_M\n)_{*R_\red})_{*R_\red},
\]
where the second equality holds because $(0:_M \n)_{*R_\red}$ is a submodule of $(0:_M\n)$, and hence its annihilator contains $\n$.

Therefore, we may assume from now on that $R$ is a reduced $F$-finite ring.  By \cite[Theorem 3.4]{HochsterHunekeTightClosureAndStrongFRegularity}, \cf  \cite[Section 6]{HochsterHunekeTC1}, there is an element $c \in \taub(R) \cap R^\circ$, which by Theorem~\ref{thm:cotestint} is a co-test element. Let $z\in M_*$.  In particular, then, $z\in M_*[c^2, 1]$.  That is, there is some $e_1$ such that there are $R$-linear maps $g_e: R^{1/q} \ra M$ for each $0 \leq e \leq e_1$ such that $z = \sum_{e=0}^{e_1} g_e(c^{2/q})$.   Now define $h_e: R^{1/q} \ra M$ via $r^{1/q} \mapsto g_e((cr)^{1/q})$.  This is clearly $R$-linear, and since $c$ is a co-test element, $\im h_e \subseteq M_*$.  In particular, $z=\sum_{e=0}^{e_1} h_e(c^{1/q}) \in (M_*)_*[c,1] = (M_*)_*$ again since $c$ is a co-test element.
\end{proof}

\begin{prop}[Minimal primes]\label{pr:minprimes}
Let $R$ be a ring such that $R_\red$ is $F$-finite, let $M$ be an $R$-module, and let $\p_1, \dotsc, \p_n$ be the minimal primes of $R$.  Then \[
M_{*R} = \sum_{i=1}^n (0 :_M \p_i)_{*(R/\p_i)}.
\]
\end{prop}

\begin{proof}
By Proposition~\ref{pr:red}, we may immediately assume that $R$ is reduced.

Next, we show that $(0 :_M \p_i)_{*(R/\p_i)} \subseteq M_{*R}$ for all $i$.  Let $\p=\p_i$ be a minimal prime and pick $z \in (0 :_M \p)_{*(R/\p)}$.  Let $c\in R^\circ$ and $e_0 \geq 0$.  Then in particular $c\in R \setminus \p$, which means that $\bar{c} \in (R/\p)^\circ$.  So there exists $e_1\geq e_0$ such that there are $(R/\p)$-linear maps $g_e: {}^e(R/\p) \ra (0:_M \p)$ for $e_0 \leq e\leq e_1$ where $z = \sum_{e=e_0}^{e_1} g_e({}^e\bar{c})$.  Consider the compositions $k_e$: \[
{}^eR \onto {}^e(R/\p) \arrow{g_e} (0 :_M \p) \hookrightarrow M.
\]
Each $k_e$ is $R$-linear and $\sum_{e=e_0}^{e_1} k_e({}^ec) = z$.  Since $c$ and $e_0$ were arbitrary, $z\in M_{*R}$.

Conversely, let $z\in M_{*R}$.  Since $R$ is reduced, we may let $c$ be an element of the conductor of $R$ that is a big test element of $R$, such that the image $\bar{c}_i$ of $c$ in $R/\p_i$ is a big test element of $R/\p_i$ for all $i$.  To find such a $c$, fix any $c' \in R^{\circ}$ such that both $R$ and each $R/\bp_i$ are regular after inverting $c'$.  We may then take $c$ to be a sufficiently large power of $c'$.  We will use the fact (Theorem~\ref{thm:cotestint}) that big test elements and co-test elements coincide.  Since $c^2$ is a co-test element of $R$, there is some $e_1\geq 0$ such that there exist $R$-linear maps $g_e: {}^eR \ra M$ such that $z = \sum_{e=0}^{e_1} g_e({}^e(c^2))$.  Let $\alpha: R \hookrightarrow \bigoplus_{i=1}^n (R/\p_i)$ be the canonical inclusion map.  Let $\beta: \bigoplus_{i=1}^n (R/\p_i) \ra R$ be the map induced by multiplication by the conductor element $c$, considering $\bigoplus_{i=1}^n (R/\p_i)$ to be a subring of the normalization of $R$.  Then $\beta \circ \alpha$ is the homothety map given by multiplication with $c$.  For each $j$, let $\gamma_j: R/\p_j \hookrightarrow \bigoplus_{i=1}^n (R/\p_i)$ be the canonical inclusion.  Let $\alpha^{(e)}$, $\beta^{(e)}$, and $\gamma_j^{(e)}$ be the corresponding maps on $p^e$th roots for each $e\leq e_1$.  Then since $\bar{c}_j$ is a co-test element of $R/\p_j$ and $g \circ \beta^{(e)} \circ \gamma_j^{(e)} \in \Hom_R({}^e(R/\p_j), M) \cong \Hom_{R/\p_j}({}^e(R/\p_j), (0:_M \p_j))$ for each $1\leq j\leq n$, we have $z_{j,e} := g(\beta^{(e)}(\gamma_j^{(e)}({}^e\bar{c}_j))) \in (0:_M \p_j)_{*(R/\p_j)}$.  Hence, \begin{align*}
z &=\sum_e g_e({}^e(c^2)) = \sum_e g_e(\beta^{(e)}(\alpha^{(e)}({}^ec))) \\
&= \sum_e g(\beta^{(e)}(\sum_{j=1}^n \gamma_j^{(e)}({}^e\bar{c}_j))) = \sum_j \sum_e z_{j,e} \in \sum_{j=1}^n (0:_M \p_j)_{*(R/\p_j)}.
\end{align*}


\end{proof}

Combined with Proposition~\ref{pr:bigtestint}, this leads to a new description of the big test ideal of certain reduced rings, which has obvious connections to the work of \cite{HochsterHunekeFRegularityTestElementsBaseChange, BravoSmithBehaviorOfTestIdealsUnderSmooth, HaraTakagiOnAGeneralizationOfTestIdeals,SchwedeTuckerTestIdealFiniteMaps,Va-testquot,Tr-diffmon}.  We will return to this issue in Section~\ref{sec:conductor}:

\begin{corollary}
Let $R$ be an $F$-finite reduced Noetherian ring of positive prime characteristic, and let $\{\p_1, \dotsc, \p_n\}$ be the minimal primes of $R$.  Then \[
\taub(R) \subseteq \sum_{i=1}^n (0 : \p_i),
\]
with equality if all of the quotient domains $R/\p_i$ are strongly $F$-regular.
\end{corollary}

\begin{proof}
$R_{*R} = \sum_{i=1}^n (0 : \p_i)_{*(R/\p_i)} \subseteq \sum_{i=1}^n (0: \p_i)$, with equality if each of these modules equals its \tint, which follows (by Proposition~\ref{pr:Fcostrong}) if $R$ is strongly $F$-regular.
\end{proof}

We now discuss a transformation rule for \tint\ under a flat morphism that sends a test element to a test element.  We can think of this as a reverse sort of persistence.

\begin{prop}\label{pr:flatbc}
Let $\phi:R \ra S$ be a flat homomorphism of reduced $F$-finite rings such that there is some $c\in R$ which is a big test element for $R$ and that $\phi(c)$ is a big test element for $S$.  Then for any $R$-module $M$, we have:
\begin{itemize}
\item[(i)] $(S \otimes_R M)_{*S} \subseteq S \otimes_R (M_{*R})$ as $S$-submodules of $S \otimes_R M$.
\item[(ii)]  If additionally we assume that for any power $q=p^e$ of $p$ the natural $S$-module map $S \otimes_R R^{1/q} \ra S^{1/q}$ is an isomorphism, then $S \otimes_R (M_{*R}) = (S \otimes_R M)_{*S}$. 
\end{itemize}
\end{prop}

\begin{proof}
First we prove (i).  
For any $\alpha \in (S \otimes_RM)_{*S}$ there exist $S$-linear maps $g_e: S^{1/q} \ra S \otimes_R M$ with $\alpha = \sum_{e=0}^{e_1} g_e(\phi(c)^{1/q})$.  Let $j_e: S \otimes_R R^{1/q} \rightarrow S^{1/q}$ be the natural map given by $s \otimes r^{1/q} \mapsto s \phi(r)^{1/q}$.  Then $\sum_e(g_e \circ j_e)(1 \otimes c^{1/q}) = \sum_e g_e(\phi(c)^{1/q}) = \alpha$.

But $g_e\circ j_e \in \Hom_S(S \otimes_R R^{1/q}, S \otimes_R M) \cong S \otimes_R \Hom_R(R^{1/q}, M)$, where the isomorphism holds because $R$ is $F$-finite and $S$ is a $R$-flat.  Thus there exist $n \in \N$, $s_{i,e} \in S$ and $h_{i,e} \in \Hom_R(R^{1/q}, M)$ such that $g_e \circ j_e \simeq \sum_{i=1}^n s_{i,e} \otimes h_{i,e}$.  Since each $h_i(c^{1/q}) \in M_{*R}$, we have \[
\alpha = \sum_e (g_e \circ j_e)(c^{1/q}) = \sum_i \sum_e s_{i,e} \otimes h_{i,e}(c^{1/q}) \in S \otimes_R M_{*R},
\]
as required.

For (ii) pick $z \in M_{*R}$.  There exist $h_e \in \Hom_R(R^{1/q}, M)$ with $z = \sum_e h_e(c^{1/q})$.  Now
\[
1 \otimes h_e \in S \otimes_R \Hom_R(R^{1/q}, M) \cong \Hom_S(S \otimes_R R^{1/q}, S \otimes_R M) \cong \Hom_S(S^{1/q}, S \otimes_R M).
\]
where the last isomorphism holds by assumption.
Hence, each $(1 \otimes h_e)(\phi(c)^{1/q}) \in (S \otimes_R M)_{*S}$, so that their sum $1 \otimes z \in (S \otimes_R M)_{*S}$.
\end{proof}

Note that the condition in (ii) above is automatically satisfied if $R \to S$ is \'etale.  Indeed, variants of the above in the context of test ideals was explored extensively in \cite{BravoSmithBehaviorOfTestIdealsUnderSmooth}.  In particular, one might expect a number of improvements to Proposition \ref{pr:flatbc} following the ideas of \cite{BravoSmithBehaviorOfTestIdealsUnderSmooth}, also see \cite[Theorem 3.3]{HaraTakagiOnAGeneralizationOfTestIdeals}.

In particular, we get corresponding results regarding localization and completion:

\begin{corollary}\label{cor:loc}
Let $R$ be an $F$-finite reduced ring, and $M$ an $R$-module.  If $R$ is local, then $\hat{R} \otimes_R (M_{*R}) \cong (\hat{R} \otimes_R M)_{*\hat{R}}$.  If $W$ is any multiplicative set, then $W^{-1} (M_{*R}) \cong (W^{-1}M)_{*(W^{-1}R)}$.
\end{corollary}

\begin{proof}
If $R$ is $F$-finite and reduced, then so is $W^{-1}R$, and it is clear that $W^{-1}R \otimes_R R^{1/q} \cong (W^{-1}R)^{1/q}$.  If $R$ is moreover local, then it follows from considering the inverse limit of the $R$-modules $R/(\m^{[q]})^n$ (for fixed $q$ and varying $n$) that $\hat{R}^{1/q} \cong \widehat{(R^{1/q})} \cong \hat{R} \otimes_R R^{1/q}$, where the second isomorphism follows from the fact that $R^{1/q}$ is finitely generated as an $R$-module.  Then the result follows from Proposition~\ref{pr:flatbc}(ii).
\end{proof}

\begin{defn}
We call a ring $R$ \emph{$F$-coregular} if for all $R$-modules $M$, $M=M_*$.
\end{defn}

In particular, the previous Proposition shows that if $R$ is $F$-coregular, so is $R_W$ for all multiplicative sets $W \subseteq R$.  The following is a strong converse:

\begin{prop}\label{pr:global}
Let $R$ be an $F$-finite reduced ring.  Suppose either that $R_\m$ is $F$-coregular for all $\m \in \Max(R)$, or that there is some set $f_1, \dotsc, f_n \in R$ such that $(f_1, \dotsc, f_n) = R$ and $R_{f_i}$ is $F$-coregular for all $1\leq i \leq n$.  Then $R$ is $F$-coregular.
\end{prop}

\begin{proof}
Let $M$ be an $R$-module.  Then we either have $(M_{*R})_\m = (M_\m)_{*R_\m} = M_\m$ for all $\m \in \Max(R)$, or $(M_{f_i})_{*R_{f_i}} = (M_{*R})_{f_i} = M_{f_i}$ for all $1\leq i \leq n$.  Since equality of modules is a local property, both with respect to localization at points and with respect to open covers, the conclusion follows.
\end{proof}

In the situations dealt with here, however, $F$-coregularity isn't really a new concept.

\begin{prop}\label{pr:Fcostrong}
Let $R$ be a ring such that $R_\red$ is $F$-finite.  Then $R$ is $F$-coregular if and only if it is strongly $F$-regular.
\end{prop}

\begin{proof}
Since both conditions imply the ring is reduced (the former because of Proposition~\ref{pr:red}), we may assume $R$ is an $F$-finite reduced ring.

If $R$ is $F$-coregular, then in particular $R = R_* = \taub(R)$, so that $R$ is also strongly $F$-regular.

Conversely, suppose $R$ is strongly $F$-regular.  Let $c$ be a big test element of $R$ and let $M$ be any $R$-module.  For some $q$ there is an $R$-linear map $g: R^{1/q} \ra R$ sending $c^{1/q} \mapsto 1$.  Let $z \in M$ and let $h: R \ra M$ be the map $r \mapsto rz$.  Then $(h \circ g) \in \Hom_R(R^{1/q}, M)$ and $(h \circ g)(c^{1/q}) = z$, so $z\in M_*$.  Hence $M=M_*$, and since $M$ was arbitrary, $R$ is $F$-coregular.
\end{proof}




\section{Co-persistence, co-contraction, co-colon capturing, and duality with tight closure}

While tight interior is a distinct notion as compared to tight closure, it has many analogous formal properties.  In this section we establish these results.  In somewhat more specialized settings, we also prove that tight interior is dual to tight closure \ref{cor.DualityWithTC}.

Persistence is one of the most important tools in any closure operation.  
  Here we discuss a dual notion in the sense of Remark \ref{rem.FormalDuality}.

\begin{lemma}[Co-persistence, first case]\label{lem:copers1}
Let $j: R \ra S$ be a ring homomorphism between not-necessarily reduced rings, $M$ an $R$-module, and consider $\Hom_R(S,M)$ as an $S$-module.  Assume either:
 \begin{itemize}
 \item that $R$ has a co-test element $c$ whose image in $S$ is not in any minimal prime of $S$, or
 \item that $j(R^\circ) \subseteq S^\circ$.
 \end{itemize}
 Then in the canonical $R$-linear evaluation map $\epsilon: \Hom_R(S,M) \ra M$ given by $\epsilon(g) := g(1)$, we have $\epsilon(\Hom_R(S,M)_{*S}) \subseteq M_{*R}$.  That is, $\epsilon$ restricts to a map $\epsilon': \Hom_R(S,M)_{*S} \ra M_{*R}$.
\end{lemma}

\begin{proof}
Let $g\in \Hom_R(S,M)_{*S}$.  Let $c$ \emph{either} be a co-test element of $R$ whose $j$-image is in $S^\circ$, \emph{or} if $j(R^\circ) \subseteq S^\circ$ then we let $c\in R^\circ$ be arbitrary.  Let $d=j(c)$ and $e_0 \geq 0$.  Since $j(c) \in S^\circ$, there is some $e_1\geq e_0$ such that there exist $S$-linear maps $\phi_e: {}^eS \ra \Hom_R(S,M)$ such that $g = \sum_{e=e_0}^{e_1} \phi_e({}^ej(c))$.  Then consider the $R$-linear compositions ${}^eR \arrow{{}^ej} {}^eS \arrow{\phi_e} \Hom_R(S,M) \arrow{\epsilon} M$.   We have \[
g(1) = \epsilon(g) = \sum_{e=e_0}^{e_1} (\epsilon \circ \phi_e \circ {}^ej)({}^ec),
\]
showing that $\epsilon(g) \in M_{*R}$, as required.
\end{proof}


The next result is dual in the sense of Remark \ref{rem.FormalDuality} to the fact that tight closure captures contractions from module finite extensions.

\begin{prop}[Co-contraction]\label{pr:cocontract}
Let $j: R \hookrightarrow S$ be a module-finite torsion-free\footnote{A map of rings $j : R \hookrightarrow S$ is called \emph{torsion-free} if $S$ is a torsion-free $R$-module.} inclusion, where $R$ is an $F$-finite domain and $S_{\red}$ is $F$-finite.  Let $M$ be an $R$-module.  Then the restricted evaluation map $\epsilon': \Hom_R(S,M)_{*S} \ra M_{*R}$ of Lemma~\ref{lem:copers1} is surjective.
\end{prop}

\begin{proof}
First we prove the proposition under the added hypothesis that $S$ is a domain.  Let $c\in R^\circ$ be a big test element shared by $R$ and $S$, the existence of which follows immediately from the fact that there is a $c \in R^{\circ} \cap S^{\circ}$ such that $R_c$ and $S_c$ are both regular.
Clearly then Lemma~\ref{lem:copers1} applies, so we get the map $\epsilon'$.  For surjectivity, let $z\in M_{*R}$.  There is a nonzero $R$-linear map $f: S \ra R$.  Let $d = f(1)$.  Then for some $e_1\geq 0$, there exist $R$-linear maps $g_e: R^{1/q} \ra M$ for $0 \leq e \leq e_1$ such that $\sum_{e=0}^{e_1} g_e((dc)^{1/q}) = z$.   Then $\sum_e (g_e \circ f^{1/q})(c^{1/q}) = z$.  Define $h_e: S^{1/q} \ra \Hom_R(S,M)$ by $h_e(s^{1/q})(t) := (g \circ f^{1/q})(t \cdot s^{1/q})$.  Then $h_e$ is $S$-linear, and letting $j_e := h_e(c^{1/q})$, we have $\epsilon(\sum_e j_e) = \sum_e j_e(1) = z$, and moreover each $j_e \in \Hom_R(S,M)_{*S}$ by construction since $c$ is a big test element for $S$. Thus $\epsilon'$ is surjective.

In the general case, we need first to establish that $\epsilon'$ exists.  Let $0\neq x\in R$.  Since $S$ is torsion-free over $R$, $j(x)$ is a non-zerodivisor of $S$, so that in particular it avoids the minimal primes of $S$, and Lemma~\ref{lem:copers1} applies to show the existence of $\epsilon'$.

Now let $P_1, \dotsc, P_n$ be the minimal primes of $S$.  Consider the maps $j_i: R \ra S/P_i$ given by composing $j$ with the natural projection $S \onto S/P_i$.  We have $\prod_i \ker j_i \subseteq \bigcap_i \ker j_i = \ker j = 0$, so that since $R$ is a domain, $\ker j_i = 0$ for some $i$.  Thus, by the domain case of the current proposition, the evaluation-at-1 map $\epsilon_i: \Hom_R(S/P_i, M) \ra M$ restricts to a surjective map $\epsilon_i': \Hom_R(S/P_i, M)_{*(S/P_i)} \onto M_{*R}$.

However, we have $\Hom_R(S/P_i, M) \cong \Hom_S(S/P_i, \Hom_R(S, M)) \cong (0 :_{\Hom_R(S,M)} P_i)$, so that by Proposition~\ref{pr:minprimes} and the above, $\epsilon_i'$ factors as $\Hom_R(S/P_i, M)_{*(S/P_i)} \hookrightarrow \Hom_R(S,M)_{*S} \arrow{\epsilon'} M_{*R}$.  Since the composition is surjective, it follows that $\epsilon'$ must be surjective.
\end{proof}

\begin{lemma}[Co-persistence, second case]\label{lem:copers2}
Suppose $R$ is an $F$-finite domain and $Q$ is a height one prime, and $M$ any $R$-module.  Then the evaluation map restricts as in Lemma~\ref{lem:copers1} when $S = R/Q$.
\end{lemma}

\begin{proof}
Let $R'$ be the normalization of $R$, let $Q'$ be a prime of $R'$ that lies over $Q$, and $T := R'/Q'$.  Then there is a big test element $c$ of $R'$ that is not in $Q'$, indeed, the big test ideal is not contained in any height-one prime since $R$ is normal and thus the singular locus of $\Spec R$ is of codimension $\geq 2$ (note the big test ideal cuts out a scheme that is trivial wherever $R$ is regular).
Then the map $R' \onto T$ satisfies the conditions of Lemma~\ref{lem:copers1}.  Moreover, the map $S \ra T$ is an injective module-finite inclusion of domains.  So the evaluation maps $\alpha: \Hom_R(R',M) \ra M$ and $\beta: \Hom_R(T,M) \ra \Hom_R(R',M)$ restrict to maps on the $(-)_*$ level (by Lemma~\ref{lem:copers1}), and the evaluation map
\[
\gamma: \Hom_{S}(T, \Hom_R(S, M)) \cong \Hom_R(S \tensor_{S} T, M) \cong \Hom_R(T,M) \ra \Hom_R(S,M)
\]
restricts to a surjection on the $(-)_*$ level (by Proposition~\ref{pr:cocontract}).  To see that the evaluation map $\delta: \Hom_R(S, M) \ra M$ restricts to a map on the $(-)_*$ level, consider the following commutative diagram: \[ \xymatrix{
\Hom_R(T, M) \ar[r]^\gamma \ar[d]^\beta & \Hom_R(S,M) \ar[d]^\delta \\
\Hom_R(R',M) \ar[r]^\alpha & M
} \]
Let $g\in \Hom_R(S,M)_{*S}$.  Then $g = \gamma(h)$ for some $h\in \Hom_R(T,M)_{*T}$.  We have $\delta(g) = \delta(\gamma(h)) = \alpha(\beta(h)) \in M_{*R}$, since both $\alpha$ and $\beta$ restrict to maps on the $(-)_*$ level.
\end{proof}

\begin{thm}[Co-persistence]\label{thm:copers}
Let $j: R \ra S$ be a ring homomorphism such that both $R_\red$ and $S_\red$ are $F$-finite.  Let $M$ be an $R$-module.  Then the natural evaluation map $\epsilon: \Hom_R(S,M) \ra M$ restricts to the $(-)_*$ level.  (That is \emph{co-persistence} holds at this level of generality for \tint s of modules.)
\end{thm}

\begin{proof}
First note that we can replace $S$ by $S/P$ for some minimal prime $P$ of $S$.  Indeed, let $\phi \in \Hom_R(S,M)_{*S}$, let $P_1, \dotsc, P_n$ be the minimal primes of $S$, and let $\pi_i: S \twoheadrightarrow S/P_i$ be the natural surjections.  Note that
\[
\begin{array}{rccl}
& \Hom_R(S/P_i, M) & \cong & \Hom_R(S \otimes_S S/P_i, M)\\
 \cong & \Hom_S(S/P_i, \Hom_R(S,M)) & \cong & (0 :_{\Hom_R(S, M)} P_i)
\end{array}
\]
and then apply Proposition~\ref{pr:minprimes}, to obtain that $\phi = \sum_{i=1}^n \phi_i \circ \pi_i$ for some elements $\phi_i \in \Hom_R(S/P_i, M)_{*(S/P_i)}$,
so that if the statement holds for all the $S/P_i$, then $\epsilon(\phi) = \phi(1) = \sum_{i=1}^n \phi_i(\bar{1}) \in M_{*R}$.

So from now on we may assume $S$ is an integral domain.  Let $Q := \ker j$, which must then be prime.
\begin{claim*}  We may replace $S$ by $\bar{R} := R/Q$.   \end{claim*}
\begin{proof}[Proof of claim]
For let $h\in \Hom_R(S,M)_{*S}$.  Take an arbitrary $\bar{c} \in \bar{R}^\circ$; then $\bar{c} \in S^\circ$ as well.  Take an arbitrary positive integer $e_0$.  Then there is some $e_1\geq e_0$ such that there exist $S$-linear maps $g_e:S^{1/q} \ra \Hom_R(S,M)$ such that $h = \sum_{e=e_0}^{e_1} g_e(\bar{c}^{1/q})$.  For each such $e$, denote the following composition by $k_e$: \[
(\bar{R})^{1/q} \hookrightarrow S^{1/q} \arrow{g_e} \Hom_R(S,M) \ra \Hom_R(\bar{R}, M),
\]
where the rightmost map is restriction.  It is clear that each $k_e$ is $\bar{R}$-linear and that $\sum_{e=e_0}^{e_1} k_e(\bar{c}^{1/q}) = h |_{\bar{R}}$, where $h|_{\bar{R}}$ denotes the image of $h$ in $\Hom_R(\bar{R}, M)$.  Since $e_0$ and $\bar{c}$ were arbitrary, it follows that $h |_{\bar{R}} \in \Hom_R(\bar{R}, M)_{*\bar{R}}$.
\end{proof}

Now that we have proved the claim, we may assume that $S=\bar{R} = R/Q$, with $j$ the natural surjection.  Take a saturated chain \[
Q_0 \subset Q_1 \subset \cdots \subset Q_\ell = Q
\]
of prime ideals in $R$, where $Q_0$ is a minimal prime of $R$.  We may replace $R$ by $R/Q_0$ because of Proposition~\ref{pr:minprimes}.  Then the conclusion follows from Lemma~\ref{lem:copers2} and induction on $\ell$, with $\ell=1$ being the base case of the induction.
\end{proof}

We now state a duality theorem with tight closure, compare with \cite[Proposition 8.23]{HochsterHunekeTC1} and the better part of \cite{SmithTestIdeals}.

\begin{prop}[Duality with tight closure, \cf \cite{SmithTestIdeals}]\label{pr:duality}
Let $R$ be an $F$-finite reduced ring, let $E$ be an \emph{injective} $R$-module, and let $(-)'$ be the contravariant functor given by $(-)' := \Hom_R(-,E)$ on the category of $R$-modules.  Let $M$ be an \emph{arbitrary} $R$-module.  Then \[
(M_*)' \cong M' / 0^*_{M'}.
\]
\end{prop}

\begin{proof}
Let $c$ be a big test element for $R$.

Let $j: M_* \hookrightarrow M$ be the canonical injection, and $j': M' \onto (M_*)'$ the corresponding surjection.  In other words, for a map $f: M \ra E$, $j'(f) := f \circ j = f |_{M_*}$.  We will show that $\ker j' = 0^*_{M'}$.

Since $E$ is injective, for \emph{finitely generated} $R$-modules $L$, there is a canonical isomorphism (see for example \cite[Chapter 2, Exercise 14, p. 45]{Bourbaki1998} or \cite[Lemma 10.2.16]{BrodmannSharpLocalCohomology}).

 \[
\alpha_L: L \otimes_R \Hom_R(M,E) \ra \Hom_R(\Hom_R(L,M), E)
\]
given by $(\alpha_L(\ell \otimes f))(g) := f(g(\ell))$ for $g \in \Hom_R(L,M)$.  We apply this below to the finite $R$-modules $R^{1/q}$.

Let $f\in M'$.  Consider what it means to say that $f\in 0^*_{M'}$.  It is equivalent to say that $c^{1/q} \otimes f = 0$ in $R^{1/q} \otimes_R M'$ for all powers $q$ of $p$.  By the isomorphisms $\alpha_{R^{1/q}}$, it is equivalent to say that $f(g(c^{1/q}))=0$ for all $R$-linear maps $g: R^{1/q} \ra M$.  But since $M_*$ is the submodule of $M$ generated by such images $g(c^{1/q})$, another equivalent condition is to say that $0 = f(M_*) = \im (f\circ j) = \im j'(f)$.  In other words, $f\in 0^*_{M'}$ if and only if $f \in \ker (j')$.  That is, \[
M' / 0^*_{M'} = M' / \ker (j') \cong \im j' = (M_*)',
\]
as was to be shown.
\end{proof}

\begin{corollary} 
\label{cor.DualityWithTC}
Let $(R,\m)$ be a complete reduced Noetherian local ring of characteristic $p$, and let $(-)^\vee$ denote the Matlis duality functor.  Let $L$ be either an Artinian or a finitely generated $R$-module (or any other Matlis-dualizable module).  Then \[
(L_*)^\vee \cong L^\vee / 0^*_{L^\vee} \quad \text{ and } \quad (L^\vee)_* \cong (L / 0^*_L)^\vee.
\]
\end{corollary}

\begin{proof}
The first statement follows from letting $M=L$ and $E = E_R(R/\m)$ in Proposition~\ref{pr:duality}.  As for the second, let $M = L^\vee$ and $E =E_R(R/\m)$ in Proposition~\ref{pr:duality}.  Then $L \cong M^\vee$ by Matlis duality, and \[
(L / 0^*_L)^\vee = (M^\vee / 0^*_{M^\vee})^\vee \cong ((M_*)^\vee)^\vee \cong M_* = (L^\vee)_*,
\]
where the last isomorphism follows from Matlis duality.
\end{proof}

Hence, we obtain immediately many statements about \tint s, at least in the complete case, by ``dualizing'' various theorems of tight closure theory.  Indeed, this is how we obtained the motivation for `co-persistence' and `co-contraction' statements.  For example, consider the following.

\begin{remark}
It would be natural to say that a $F$-finite ring $R$ is \emph{$F$-corational} if it is Cohen-Macaulay and $(\omega_R)_* = \omega_R$.  This concept already coincides with the definition of $F$-rationality in this context as can be readily verified, so we will say no more about it.
\end{remark}

Colon capturing is an extremely useful property of tight closure, and so we should expect that interesting dual statements can be made with respect to \tint\ and test ideals.  First we explain what a dual version of colon capturing would be.  In fact, the following is essentially dual to \cite[Theorem 3.1A]{HunekeTightClosureBook}, which says that if $R$, $A$ are as given below and $I$, $J$ are ideals of $A$, then $(IR :_R JR) \subseteq ((I :_A J)R)^*$.

\begin{thm}[Co-colon capturing]\label{thm.CoColonCapturing}
Let $A$ be a regular $F$-finite ring, and let $R$ be a module-finite and torsion-free ring extension of $A$.  Let $M$ be an $A$-module and $x\in A$.  Then \[
\Hom_A(R, xM)_{*R} = x \cdot \Hom_A(R,M)_{*R},
\]
where the module $\Hom_A(R,xM)$ is considered as a submodule of $\Hom_A(R,M)$ by the left-exactness of $\Hom$.
\end{thm}

\begin{proof}
It is easy to show that $x \cdot \Hom_A(R,M)_{*R} \subseteq \Hom_A(R, xM)_{*R} $.  Indeed, this direction is true for any ring homomorphism $A \ra R$, regardless of the properties of $A$, $R$, or the homomorphism.

Conversely, let $g\in \Hom_A(R, xM)_{*R}$.  There is a finite free $A$-submodule $G$ of $R$ such that $R/G$ is a torsion $A$-module.  Take a nonzero element $c$ of $A$ such that $cR \subseteq G$. Let $d\in R^\circ$ and let $q_0$ be a power of $p$.

Then $cd \in R^\circ$, so by the definition of \tint, there is some power $q_1\geq q_0$ of $p$ and $R$-linear maps $\mu_e: {}^eR \ra \Hom_A(R, xM)$ such that $\sum_{e=e_0}^{e_1} \mu_e({}^e(cd))= g$.  Each $\mu_e$ then induces an $A$-linear map $\alpha_e: {}^eR \ra xM$ (namely, $\alpha_e({}^er) = \mu_e({}^er)(1)$ for all ${}^er \in {}^eR$), so that
\[
\sum_{e=e_0}^{e_1} \alpha_e({}^e(cd)) = \sum_{e=e_0}^{e_1} \mu_e({}^e(cd))(1) = g(1),
\]
and more generally, for any $y \in R$, we have
\[
\sum_{e=e_0}^{e_1} \alpha_e(y \cdot {}^e(cd)) = \sum_{e=e_0}^{e_1} \mu_e(y \cdot {}^e(cd))(1) = \sum_{e=e_0}^{e_1} \mu_e({}^e(cd))(y) = g(y).
\]
Let $\beta_e$ be the restriction of  $\alpha_e$ to the $A$-submodule ${}^eG$ of ${}^eR$; note that ${}^e(cd) \subseteq {}^eG$.  Then $\beta_e \in \Hom_A({}^e G, xM) = x \Hom_A({}^e G, M)$, where equality holds because ${}^e G$ is a finite free $A$-module (since $A$ is regular and $F$-finite).  That is, $\beta_e = x \gamma_e$ for some $A$-linear $\gamma_e: {}^e G \ra M$. One then obtains $A$-linear maps $\delta_e: {}^eR \ra M$ by setting $\delta_e({}^er) = \gamma_e({}^e(cr))$ (which is well-defined since $cR \subseteq G$).  So $\delta_e \in \Hom_A({}^eR, M) \cong \Hom_R({}^eR, \Hom_A(R,M))$, and if we let the image of $\delta_e$ under this isomorphism be $\epsilon_e: {}^eR \ra \Hom_A(R,M)$, then we have for any $y \in R$,
\begin{align*}
\left(x \cdot \sum_{e=e_0}^{e_1} \epsilon_e({}^ed)\right)(y) &=  \sum_{e=e_0}^{e_1} (x \delta_e)(y \cdot {}^ed) =  \sum_{e=e_0}^{e_1} (x \gamma_e)(y \cdot {}^e(cd)) \\
&=  \sum_{e=e_0}^{e_1} \beta_e(y \cdot {}^e(cd)) =  \sum_{e=e_0}^{e_1} \alpha_e(y \cdot {}^e(cd)) = g(y).
\end{align*}
Because $y \in R$ was arbitrary, it follows that $g = x \cdot \sum_e \epsilon_e({}^ed) \in x (\Hom_A(R,M)_{*R})$.
\end{proof}

\section{Cophantom resolutions}
\label{sec.Cophantom}
Let $R$ be an $F$-finite ring of characteristic $p > 0$ and $N$ an $R$-module.  One way to turn $N$ into an ${}^eR$-module is by the tensor product (in other words, the Peskine-Szpiro Frobenius functor, \cite{PeskineSzpiroDimensionProjective}).  However, there is a dual approach.  We define
\[
\NPullback{e}{R}(N) := \Hom_R({}^eR, N).
\]
with the left-module structure coming from the right-module structure of ${}^eR$.  If $R$ is reduced, this is the same as first taking the natural $R^{1/q}$-module structure on $\Hom_R(R^{1/q}, N)$, then viewing it as an $R$-module via the isomorphism $R \cong R^{1/q}$ sending each $a \mapsto a^{1/q}$.
In the geometric setting $\NPullback{e}{R}(N)$ is simply $\myH^0((F^e)^! N)$, the zeroth cohomology of the $e$-iterated Frobenius upper-shriek, see \cite[Chapter III, Section 6]{HartshorneResidues}.

Several properties of the functor $\NPullback{e}{R}(-)$ are developed by J. Herzog in \cite{Her-Frob}, where it is denoted $\tilde{F}^e(-)$:

\begin{lemma}[{\cite[from 2.7]{Her-Frob}}]\label{lem:Herinj}
For any injective $R$-module $N$, $\NPullback{e}{R}(N)$ is also injective.
\end{lemma}

\begin{lemma}[{\cite[Lemma 4.1]{Her-Frob}}] \cf \cite[Chapter 3]{BrunsHerzog}, \cite{HartshorneResidues}
If $(R,\m, k, E)$ is a local $F$-finite ring, then $\NPullback{e}{R}(E) \cong E$.
\end{lemma}

\begin{thm}[{\cite[from Satz 5.2]{Her-Frob}}]\label{thm:HerFrob}
Let $R$ be a local $F$-finite ring and $M$ a finitely generated $R$-module.  Then $M$ has finite injective dimension if and only if $\Ext^i_R({}^eR, M) =0$ for all $i, e>0$.
\end{thm}

We will also find the following property very useful:
\begin{prop}[Co-base change of $\text{\Fontskrivan\slshape{F}}_e$]\label{pr:cobase}
Let $R \ra S$ be a map of $F$-finite rings.  Then $\NPullback{e}{S}(\Hom_R(S,-)) \cong \Hom_R(S, \NPullback{e}{R}(-))$.
\end{prop}

\begin{proof}
This can be thought of as a special case of the fact that the formation of upper-shriek respects compositions (and the fact that $R \to S \to {}^e S$ is the same as $R \to {}^e R \to {}^e S$) \cf \cite[Chapter III, Proposition 6.6]{HartshorneResidues}.

To see it algebraically, we have \begin{align*}
\Hom_S({}^eS, \Hom_R(S, -)) &\cong \Hom_R({}^eS, -) \cong \Hom_R({}^eR \otimes_R S, -) \\
&\cong \Hom_R(S, \Hom_R({}^eR, -)).
\end{align*}
\end{proof}

Let $C^{\mydot}$ be a complex of $R$-modules indexed cohomologically.  We say that $C^{\mydot}$ has \emph{cophantom cohomology at $i$} if $Z^i(C^{\mydt})_* \subseteq B^i(C^{\mydt})$.  We say it has \emph{stably} cophantom cohomology if this is also true for the induced complex $\Hom_R({}^eR, C^{\mydt})$ for all $e$ (the $R$-module action on this complex is on the left, and so it is equivalent to taking the $R^{1/q}$-module interior of $\Hom_R(R^{1/q}, C^{\mydt})$).

  For an $R$-module $M$, a complex $E^{\mydot}$ of injective modules is a \emph{cophantom resolution} of $M$ if $E^i = 0$ for $i<0$, $H^0(E^{\mydt}) \cong M$, and $E^{\mydot}$ has stably cophantom cohomology at every $i > 0$.  The length of the shortest possible cophantom resolution of $M$ is called the \emph{cophantom injective dimension} (cid) of $M$.  If there is no such resolution, we say that cid$(M) = \infty$.

\begin{remark}
\label{rem.FiniteIDGivesFiniteCID}
For instance, if $M$ is finitely generated and has finite injective dimension, any injective resolution is a cophantom resolution as well, which means that cid$(M) \leq $id$(M)$. 
To see this, let $E^\mydot$ be an injective resolution of $M$.  By left-exactness of $\Hom$, it is clear that $\NPullback{e}{R}(M) = H^0(\NPullback{e}{R}(E^{\mydt}))$, and we know that each $\NPullback{e}{R}(E^i)$ is injective, so it suffices to show that $\NPullback{e}{R}(E^\mydt)$ is acyclic.  But for any $i>0$ and $e>0$, \[
H^i(\NPullback{e}{R}(E^\mydt)) = H^i(\Hom_R({}^eR, E^\mydt)) = \Ext^i_R({}^eR, M) = 0,
\]
by Theorem~\ref{thm:HerFrob}.
\end{remark}

\begin{lemma}\label{lem:phantom}
Let $R \hookrightarrow S$ be a module-finite torsion-free extension of $F$-finite rings.   Assume either: \begin{itemize}
\item that $R$ has a co-test element $c$ that is not in any minimal prime of $S$, or
\item that $R^\circ \subseteq S^\circ$.
\end{itemize}
Let \[
\theta: (U \arrow{\alpha} V \arrow{\beta} W)
\]
be a sequence of $R$-modules and homomorphisms such that \begin{enumerate}
\item $\alpha$ is injective,
\item $\beta \circ \alpha = 0$, and
\item $(\ker \beta)_* \subseteq \im \alpha$.
\end{enumerate}
Then the sequence $\Hom_R(S, \theta)$ of $S$-modules has the same three properties.
\end{lemma}

\begin{proof}
Parts 1 and 2 are obvious, since $\Hom_R(S,-)$ is a left-exact  additive functor.

As for 3, let $\alpha': \Hom_R(S, U) \ra \Hom_R(S,V)$ and $\beta': \Hom_R(S,V) \ra \Hom_R(S,W)$ be the $S$-linear maps induced from $\alpha$, $\beta$ respectively.  Since $S$ is finitely presented as an $R$-module, we have a finite $R$-free presentation $\omega: (R^n \ra R^m \ra S \ra 0)$.  Then we get a double complex $\Hom_R(\omega, \theta)$ which is represented as the following commutative diagram with exact rows:
\[ \xymatrix{
0 \ar[r]& \Hom_R(S,U) \ar[r]^-{d} \ar@{^{(}->}[d]^{\alpha'} & U^m \ar@{^{(}->}[d]^{\alpha^m} \ar[r]^{g} & U^n \ar@{^{(}->}[d]^{\alpha^n} \\
0 \ar[r]& \Hom_R(S,V) \ar[r]^-{s} \ar[d]^{\beta'} & V^m \ar[r]^h \ar[d]^{\beta^m} &V^n \ar[d]^{\beta^n}\\
0 \ar[r]& \Hom_R(S,W) \ar[r]^-{f} & W^m \ar[r]^i & W^n
} \]
We want to show that $\im \alpha' \supseteq (\ker \beta')_{*S}$.  So let $\phi \in (\ker \beta')_{*S}$.
We have $0=f(0) = f(\beta'(\phi)) = \beta^m(s(\phi))$, so that $s(\phi) \in \ker \beta^m$.  Moreover, let $\delta: \ker \beta' \hookrightarrow \Hom_R(S,V)$ be the natural injection.  Let $c$ be a co-test element of $R$ that is not in any minimal prime of $S$, or if $R^\circ \subseteq S^\circ$, let $c \in R^\circ$ be arbitrary.  Take any power $q_0$ of $p$.  Since $c \in S^\circ$, there is some $e_1 \geq e_0$ and $S$-linear maps $\gamma_e: S^{1/q} \ra \ker \beta'$ such that $\sum_{e=e_0}^{e_1} \gamma_e(c^{1/q}) = \phi$.  Let $j_e: R^{1/q} \hookrightarrow S^{1/q}$ be the induced structure map for each $e$.  Then each $s\delta \gamma_e j_e: R^{1/q} \ra V^m$ is $R$-linear, and $\beta^m s \delta \gamma_e j_e = f \beta' \delta \gamma_e j_e = 0$, so that $\im (s \delta \gamma_e j_e) \subseteq \ker \beta^m$.  So we obtain $R$-linear maps $\epsilon_e: R^{1/q} \ra \ker \beta^m$ such that $s\delta \gamma_e j_e = (\ker \beta^m \hookrightarrow V^m) \circ \epsilon_e$, and so $\sum_{e=e_0}^{e_1} \epsilon_e(c^{1/q}) = s(\phi)$.  Thus, $s(\phi) \in (\ker \beta^m)_{*R} = ((\ker \beta)_{*R})^{\oplus m} \subseteq (\im \alpha)^{\oplus m} = \im \alpha^m$.  The following claim finishes the proof
\begin{claim}
$\phi \in \im \alpha'$
\end{claim}
\begin{proof}[Proof of claim]
There is some $x\in U^m$ such that $s(\phi) = \alpha^m(x)$.  Then \[
0 = h(s(\phi)) = h(\alpha^m(x)) = \alpha^n(g(x)),
\]
and since $\alpha^n$ is injective, $x \in \ker g = \im d$.  Thus, $x=d(\psi)$ for some $R$-linear $\psi: S \ra U$.  So we have $s(\phi) = \alpha^m(x) = \alpha^m(d(\psi)) = s(\alpha'(\psi))$.  Since $s$ is injective, $\phi = \alpha'(\psi) \in \im \alpha'$.
\end{proof}\end{proof}

As a result, we get the following:
\begin{prop}\label{pr:coph}
Let $R \ra S$ be as in Lemma~\ref{lem:phantom}.  Let $M$ be an $R$-module that has a cophantom injective resolution $E^{\mydot}$ over $R$.  Then $\Hom_R(S, E^{\mydt})$ is a cophantom injective resolution of $\Hom_R(S,M)$ over $S$.
\end{prop}

\begin{proof}
Label the maps in the cophantom injective resolution $\delta^i: E^i \ra E^{i+1}$.  Then we get the following sequences for each $i$: \[
\theta_i: (\im \delta^{i-1} \arrow{\alpha_i} E^i \arrow{\beta_i} \im \delta^i)
\]
Then each $\theta_i$ satisfies the conditions (hence also the conclusion) of the $\theta$ in Lemma~\ref{lem:phantom}.  The fact that it is stably cophantom follows similarly now using Lemma~\ref{lem:Herinj} and Proposition~\ref{pr:cobase}.
\end{proof}

\begin{corollary}\label{cor:coph}
Let $R \ra S$ be as above, and let $M$ be a finite $R$-module that has a finite $R$-injective resolution $E^{\mydot}$.  Then $\Hom_R(S,E^{\mydt})$ is a finite $S$-cophantom injective resolution of $\Hom_R(S,M)$.
\end{corollary}

\begin{proof}
By Remark \ref{rem.FiniteIDGivesFiniteCID} (see Theorem~\ref{thm:HerFrob}), $\NPullback{e}{R}(E^\mydt)$ is an injective resolution of $\NPullback{e}{R}(M)$, hence a stably cophantom acyclic complex, so that $E^\mydot$ is a cophantom injective resolution of $M$ over $R$.  Then apply Proposition~\ref{pr:coph}.
\end{proof}

\begin{thm}[Vanishing theorem for maps of Ext]
\label{thm:VanishingForExt}
Consider a sequence of ring homomorphisms $A \hookrightarrow R \ra T$ such that $T$ is $F$-finite and regular (or simply strongly $F$-regular), $R$ is a module-finite torsion-free extension of $A$, $A$ is a domain, and both $A$ and $R$ are $F$-finite.  Let $M$ be a finite $A$-module of finite injective dimension.  Then for all $i\geq 1$, the natural maps $\Ext^i_A(T, M) \ra \Ext^i_A(R, M)$ are zero.
\end{thm}

\begin{proof}
Let $(E^\mydt, \partial^\mydt)$ be a finite $A$-injective resolution of $M$.  For each $i > 0$, we have the induced maps $\partial^i_T : \Hom_A(T, E^i) \ra \Hom_A(T, E^{i+1}))$ and $\partial^i_R : \Hom_A(R, E^i) \ra \Hom_A(R, E^{i+1}))$.  Let $\eta \in \ker \partial^i_T$, and let $g^i: \Hom_A(T, E^i) \ra \Hom_A(R, E^i)$ be the map given by restriction.  Then since $\eta \in (\ker \partial^i_T)_{*T}$ since $T$ is $F$-coregular (Proposition~\ref{pr:Fcostrong}), we claim that Theorem~\ref{thm:copers} (co-persistence) shows that $g^i(\eta) \in (\ker \partial^i_R)_{*R}$.  We now justify this claim:

We have a commutative diagram with exact rows
\[
\xymatrix{
0 \ar[r] & \ker \partial^i_R \ar[r] & \Hom_A(R, E^i) \ar[r] & \Hom_A(R, E^{i+1}) \\
0 \ar[r] & \Hom_R(T, \ker \partial^i_R) \ar[u] \ar[r] & \Hom_R(T, \Hom_A(R, E^i)) \ar[r] \ar[u]  & \Hom_R(T, \Hom_A(R, E^{i+1})) \ar[u]  \\
0 \ar[r] & \ker \partial^i_T \ar[r] \ar[u]^{\nu} & \Hom_A(T, E^i) \ar[r] \ar[u]_{\sim} & \Hom_A(T, E^{i+1}) \ar[u]_{\sim}
}
\]
where the vertical isomorphisms are due to adjointness of $\Hom$ and tensor and so $\nu$ is an isomorphism also.  Notice furthermore that the middle vertical composition is just $g^i$.  The claim then immediately follows by co-persistence.


But by Corollary~\ref{cor:coph}, $\Hom_A(R,E^\mydt)$ is a stably cophantom injective resolution of $\Hom_A(R,M)$, which implies that $(\ker \partial^i_R)_{*R} \subseteq \im \partial^{i-1}_R$.  Thus, $g^i(\eta) = \partial^i_R(\mu)$ for some $\mu \in \Hom_A(R, E^{i-1})$.  Thus, the class of $g^i(\eta)$ is zero in $\Ext^i_A(R, M)$, completing the proof.
\end{proof}

\section{$F$-pure Cartier modules}
\label{sec.CartierModules}

Suppose that $R$ is an $F$-finite ring and $M$ is a \emph{finite} $R$-module.  Recently, M.~Blickle developed a theory of an ``algebra of $p^{-e}$-linear maps'' acting on $M$ \cite{BlickleTestIdealsViaAlgebras} and \cf \cite{SchwedeTestIdealsInNonQGor}.
Indeed, consider the object:
\[
\sC_M = R \oplus \left( \bigoplus_{e = 1}^{\infty} \Hom_R({}^e M, M) \right) = \oplus_{i \geq 0} \sC^i_M
\]
This is called the \emph{full Cartier algebra on $M$}.  It is a non-commutative graded ring (the first direct summand is degree zero) with multiplication defined by the following rule:
For $\phi \in \sC_M^e$ and $\psi \in \sC_M^d$, we define
\[
 \phi \cdot \psi := \phi \circ ({}^e \psi)
\]
More generally, Blickle considered graded subrings of $\sC_M$.  In this paper, we limit ourselves to the canonical choice of $\sC_M$.

\begin{remark}
  Alternatively, it just as natural to consider the graded ring $\sB_{M} =  \bigoplus_{i = 0}^{\infty} \Hom_R({}^e M, M)$ which only differs from $\sC_M$ only in the degree zero piece and so may in this sense include endomorphisms of $M$ not coming from multiplication.
\end{remark}

With notation as above, we set $\sC_M^{+} = \bigoplus_{i = 1}^{\infty} \Hom_R({}^e M, M)$ to be the positively graded part of $\sC_M$.  Given any submodule $N \subseteq M$, we define
\[
 \sC_M^{+} \cdot N := \sum_{e \geq 1} \sum_{\phi \in \sC_M^e} \phi({}^e N).
\]

\begin{defn}
 Given $F$-finite $R$, a finite $R$-module $M$ and $\sC_M$ as above, we say that a submodule $N \subseteq M$ is \emph{$\sC_M$-compatible} if $\phi({}^e N) \subseteq N$ for all $\phi \in \sC_M^e$ and all $e \geq 0$.  In other words, $\sC_M^+ \cdot N \subseteq N$.  We say that $N$ is \emph{$\sC_M$-fixed} if $\sC_M^{+} \cdot N = N$.
\end{defn}

In \cite{BlickleBoeckleCartierModulesFiniteness} and \cite{BlickleTestIdealsViaAlgebras}, many remarkable properties of $\sC_M$-fixed modules are studied.  First we recall a theorem which allows us to associate a fixed submodule to any compatible submodule.

\begin{thm} \cite[Proposition 2.13]{BlickleTestIdealsViaAlgebras} \cf \cite{Gabber.tStruc,LyubeznikFModulesApplicationsToLocalCohomology,HartshorneSpeiserLocalCohomologyInCharacteristicP}
 Given any $\sC_M^{R}$-compatible submodule $N \subseteq M$, the descending chain:
\[
 N \supseteq \sC_M^{+} \cdot N \supseteq \sC_M^{+} \cdot (\sC_M^{+} \cdot N) \supseteq \sC_M^{+}\cdot (\sC_M^{+}\cdot (\sC_M^{+} \cdot N)) \supseteq \dots
\]
stabilizes.  We use $\underline{N}$ to denote the stable term (note that $\underline{N}$ is by definition $\sC_M$-fixed).
\end{thm}

With this theorem, Blickle made the following definition.
\begin{defn}
 Given $R$ and $M$ as above, we define $\tau(M, \sC_M)$, if it exists, to be the unique smallest submodule of $M$ agreeing with $\underline{M}$ at the generic points (\ie minimal associated primes) of $M$.
\end{defn}

It is unclear that $\tau(M, \sC_M)$ exists and indeed, this is an open question in general (unless $R$ is finite type over a field \cite[Theorem 4.13]{BlickleTestIdealsViaAlgebras}).
Our goal is to relate $\tau(M, \sC_M)$ to $M_*$ under some mild hypotheses.

\begin{lemma}
\label{lem.InteriorIsCompatible}
Suppose that $R$ is $F$-finite and has a co-test element.
Then the submodule $M_* \subseteq M$ is a $\sC_M$-compatible.
\end{lemma}
\begin{proof}
Let $b\geq 0$, choose $\psi \in \Hom_R({}^b M, M)$ and write $M_* = \sum_{e \geq 0} \sum_{\phi \in \Hom_R({}^eR, R)} \phi({}^e(cR))$ for some co-test-element $c \in R^{\circ}$.  Then
\[
\begin{array}{rl}
& \psi({}^bM_*)\\
 =&  \psi\left({}^b\left( \sum_{e \geq 0} \sum_{\phi \in \Hom_R({}^e R, M)} \phi({}^e (d R)) \right)\right)\\
 = & \sum_{e \geq 0} \sum_{\phi \in \Hom_R({}^e R, M)} \psi({}^b \phi({}^e (c R)))\\
\subseteq & \sum_{e \geq b} \sum_{\phi \in \Hom_R({}^e R, M)} \phi({}^e (cR))\\
\subseteq & M_*
\end{array}
\]
\end{proof}

It immediately follows that $\tau(M, \sC_M) \subseteq M_*$ if we know that $M_*$ generically agrees with $\underline{M}$.  Indeed, if $R$ is reduced and $M$ has the same support as $R$, then it is easy to see that generically $M_*$ agrees with $M$ which agrees with $\tau(M, \sC_M)$.

\begin{thm}
\label{thm.InteriorVsBlickle}
Let $R$ be an $F$-finite reduced ring and $M$ a finite $R$-module whose support is equal to the support of $R$.  Then $M_* = \tau(M, \sC_M)$.  In particular, $\tau(M, \sC_M)$ exists.
\end{thm}
\begin{proof}
Fix $c \in R^{\circ}$ a big test-element so that
\[
M_* = M^c := \sum_{e \geq 0} \sum_{\phi \in \Hom_R({}^eR, R)} \phi({}^e(cR)).
\]
Suppose first that $N$ is any $\phi$-compatible submodule such that $\underline{N}$ agrees generically with $\underline{M}$ (which automatically generically agrees with $M$ since we are using the full algebra $\sC_M$) and choose $d \in R^{\circ}$ such that $d \underline{M} \subseteq dM \subseteq \underline{N} \subseteq N$.  It follows that \[
\sum_{e \geq 0} \sum_{\phi \in \sC_M^e} \phi({}^e( c^3 dM) ) \subseteq N.
  \]
Fix $\p_1, \dots, \p_t \subseteq R$ to be the minimal primes of $R$ and notice that $c \cdot (\oplus_i R/\p_i) \subseteq R$ since $c$ is a test element.  By our assumption on the support of $M$, we know that $M_{\p_i}$ is a non-zero $R_{\p_i} = R_{\p_i}/\p_i R_{\p_i}$-vector space.  We then claim we obtain an inclusion map
\[
G := \oplus_{i=1}^t (R/ \p_i)^{\oplus n_i} \xrightarrow{\beta} M
\]
which is an isomorphism at the $R_{\p_i}$.

To see this, for each $i = 1, \dots, t$ set $\ba_i = \prod_{j \neq i} \bp_j$ and first observe that $\ba_i \cdot M_{\bp_i} = M_{\bp_i}$ since the $\bp_i$ are minimal.  Now, $\ba_i \cdot M_{\bp_i} = M_{\bp_i}$ is a finite dimensional $R_{\bp_i}$-vector space, and so we may choose elements $m_{i, 1}, \dots, m_{i, n_i} \in\ba_i \cdot M$ whose image in $M_{\bp_i}$ form a basis.  Notice that if $i \neq j$, then the image of $m_{i, s}$ in $M_{\bp_j}$ is zero.
Consider the map $\delta : \oplus_{i=1}^t (R/ \p_i)^{\oplus n_i} \to M$ in which each standard basis element is sent to $m_{i, j}$.  This is clearly generically an isomorphism (although it need not be surjective non-generically).  On the other hand we now consider injectivity.  Any element of $\oplus_{i=1}^t (R/ \p_i)^{\oplus n_i}$ is non-zero after localizing at some $\p_i$, and so if $\delta$ sends an element to zero, then since $\delta_{\p_i}$ is an isomorphism, the element had to have been zero to begin with.  This completes the proof that $\delta$ is injective.

Set the cokernel of $\beta$ to be $C$.  Choose $x \in R^{\circ}$ that annihilates $C$, then $x \cdot \Ext^1({}^e C, M) = 0$ for all $e \geq 0$ as well.  Set $H_e = \Hom_R({}^e G, M)$.  Since we have the exact sequence
\[
\dots \to \Hom_R({}^e M, M) \to H_e = \Hom({}^e G, M) \to \Ext^1({}^e C, M) \to \dots
\]
and $x$ annihilates  $\Ext^1({}^e C, M)$, we immediately see that every element of $x \cdot H_e = x \cdot \Hom_R({}^e \left( \oplus_{i=1}^t (R/\p_i)^{\oplus n_i} \right), M)$ extends to an element of $\sC_M^e = \Hom_R({}^e M, M)$.


Now observe that
\[
\begin{array}{rcl}
&  & \sum_{e \geq 1} \sum_{\phi \in x \cdot H_e} \phi\left({}^e (c^3d (\oplus_i (R/\p_i)^{\oplus n_i})))\right) \\
& = & \sum_{e \geq 1} \sum_{\phi \in H_e} \phi\left({}^e (c^3 dx (\oplus_i (R/\p_i)^{\oplus n_i})))\right) \\
& = & \sum_{e \geq 1} \sum_{i=1}^t \sum_{\phi_i \in \Hom_R({}^e R/\p_i, M)} \phi_i\left({}^e (c^3 dx R/\p_i))\right) \\
& = & \sum_{e \geq 1} \sum_{\phi \in \Hom_R({}^e (\oplus_i R/\p_i), M)} \phi\left({}^e (c^2 dx R))\right)\\
& = & \sum_{e \geq 1} \sum_{\phi \in \Hom_R({}^e (c \cdot \oplus_i R/\p_i), M)} \phi\left({}^e (c dx R))\right)\\
& \supseteq & \sum_{e \geq 1} \sum_{\phi \in \Hom_R({}^e R, M)} \phi({}^e( cdx R))\\
& = & M_*
\end{array}
\]
Putting this together, we obtain
\[
M_* \subseteq \sum_{e \geq 0} \sum_{\phi \in x \cdot H_e} \phi\left({}^e ( c^3d \bigoplus_{i} (R/\p_i)^{\oplus n_i})\right)) \subseteq \sum_{e \geq 0} \sum_{\phi \in \sC_M^e} \phi({}^e ( c^2dM)) \subseteq N
\]
as desired.
\end{proof}

\begin{remark}
The condition that $M$ has the same support as $R$ is needed because in M.~Blickle's definition of $\tau(M, \sC_M)$, the minimal primes that matter are the minimal primes of the support of $M$.  In tight closure theory, the minimal primes that matter are the minimal primes of $R$.  Thus in order to make these notions coincide, we need to line these primes up.  One could modify the definition $\tau(M, \sC_M)$ to be the smallest module coinciding with $\underline{M}$ at the minimal primes of $R$ and obtain more general versions of the above result.
\end{remark}

\section{Transformation rules for interiors}
\label{sec.TransformationRulesForInteriors}

In this section we prove additional transformation rules for tight interiors (and so in particular for big test ideals) under ring maps, Corollary \ref{cor.TauTransformsUnderFiniteExtensions}.  This is a corollary of Theorem \ref{thm:copers} (co-persistence) and the theory of Cartier modules developed in Section \ref{sec.CartierModules}.  In the special case that of tight interior of rings ({\itshape{i.e.}} for test ideals), this result is both complementary to, and subsumes special cases of, the main results of \cite{SchwedeTuckerTestIdealFiniteMaps}.  In particular, from this perspective it seems that the transformation rules for test ideals described in \cite{SchwedeTuckerTestIdealFiniteMaps} should be viewed as a sort of persistence.  It would be interesting to develop a theory which contains both of these results as corollaries.


\begin{prop}
\label{prop.StableTakenToStable}
Suppose that $R$ is a ring of characteristic $p > 0$ such that and $R \subseteq S$ an extension.  Suppose that $M$ is an $S$-module and $N \subseteq M$ is $\sC_{M}$-compatible (in other words, for every $S$-linear map $\psi : {}^eM \to M$, $\psi({}^eN) \subseteq N$).  Fix an $R$-module $L$ and consider the $R$-submodule of $L$
\[
E := \sum_{e \geq 0} \sum_{\phi} \phi( {}^e N )
\]
where the inner sum runs over $\phi \in \Hom_R( {}^eM, L)$.  Then $E \subseteq L$ is $\sC_{L}$-compatible.
\end{prop}
\begin{proof}
Fix any $R$-linear map $\beta : {}^dL \to L$, then
\[
\begin{array}{rcl}
\beta({}^dE) & = & \beta \left( {}^d \left( \sum_{e \geq 0} \sum_{\phi} \phi( {}^eN ) \right) \right)\\
 & = & \sum_{e \geq 0} \sum_{\phi} \beta \left( {}^d \phi( {}^eN ) \right)\\
 &  = & \sum_{e \geq 0} \sum_{\phi} \beta \circ ({}^d\phi)( {}^{e+d} N )\\
 & \subseteq & E
\end{array}
\]
where again $\phi$ runs over $\Hom_R( {}^eM, L)$.
\end{proof}

\begin{corollary}
\label{cor.PropEasyContainmentOfTauForFinite}
Assume that $R \subseteq S$ is a finite extension of $F$-finite reduced rings.  Additionally suppose that $L$ is a finite $R$-module whose support equals $\Spec R$ and $M$ is a finite $S$-module whose support equals $\Spec S$, then
\[
L_{*R} \subseteq \sum_{e \geq 0} \left( \sum_{\phi \colon {}^eM \to L} \phi( {}^e (M_{*S}) ) \right).
\]
\end{corollary}
\begin{proof}
By Theorem \ref{thm.InteriorVsBlickle} $L_{*R}$ is the unique smallest submodule of $L$ which agrees with $L$ at the generic points of $R$ and which is $\sC_{L}$-compatible.  Therefore, we merely need to see that the submodule $E$ defined above in Proposition \ref{prop.StableTakenToStable} also agrees with $L$ at the generic points of $\Spec R$.

Observe that if $\eta$ is a generic point of $\Spec R$, then $S_{\eta}$ is a finite direct sum of fields and thus $(M_*)_{\eta} = M_{\eta}$.
Thus we may assume that $R$ is a field, $S$ is a finite direct sum of finite extension fields and $L$ and $M$ are finite $R$ and $S$-modules, respectively.  Then $E := \sum_{e \geq 0} \sum_{\phi \colon {}^eM \to L} \phi( {}^e M )$ is clearly equal to $L$ since ${}^e M$ and $L$ are both finite dimensional $R$-vector spaces.
\end{proof}

We first need the following lemma.

\begin{lemma}
\label{lem.MInteriorPartiallyCommutesWithFrobenius}
For any $R$-module $M$ we have a containment ${}^e(M_*) \subseteq ({}^e M)_*$.
\end{lemma}
\begin{proof}
Choose $z \in M_*$.
For any $c \in R^{\noMinPrime}$ and $e_0 > 0$, there exists $e_1, \dots, e_n > e_0$ and $\phi_i \in \Hom_{R}({}^{e_i} R, M)$ such that $z = \sum_{i = 1}^n \phi_i({}^{e_i} c)$.  Thus ${}^ez = \sum_{i = 1}^n ({}^e\phi_i)({}^{e+e_i} c)$.  Note that ${}^e \phi_i \in \Hom_{{}^eR}({}^{e_i+e} R, {}^eM)$.  But this immediately implies that ${}^ez \in ({}^e M)_*$.
\end{proof}

Combining this with co-persistence, we obtain the following persistence-like statement which is interesting on its own.

\begin{prop}
\label{prop.InteriorTransformsUnderExtensions}
Suppose that $R$ is a ring of characteristic $p > 0$ such that $R_{\red}$ is $F$-finite and $R \to S$ is a ring map with $S_{\red}$ also $F$-finite.  Fix $M$ to be an $S$-module and $L$ to be an $R$-module.  Then:
\begin{equation}
\label{eq.TauDescriptionForFiniteMaps}
L_{*R} \supseteq \sum_{e \geq 0} \sum_{\phi} \phi( {}^e (M_{*S}) )
\end{equation}
where $\phi$ ranges over all elements of $\Hom_{R}( {}^e M, L)$.
\end{prop}
\begin{proof}
First note that $\Hom_R({}^e M, L) \cong \Hom_S({}^e M, \Hom_R(S, L))$.  Consider now $\phi \in \Hom_R({}^e M, L)$ with induced $\phi' \in \Hom_S({}^e M, \Hom_R(S, L))$.  It follows that $\phi( {}^e (M_{*S}) ) = \varepsilon\big( \phi'( {}^e (M_{*S}) ))\big)$ where $\varepsilon : \Hom_R(S, L) \to L$ is the ``evaluation-at-1'' map.  By Lemma \ref{lem.MInteriorPartiallyCommutesWithFrobenius} above, we obtain that $\varepsilon\big( \phi'( {}^e (M_{*S}) ))\big) \subseteq \varepsilon\big( \phi'( ({}^e M)_{*S} ))\big)$.  Now applying Lemma \ref{lem.ModuleMapsRespectInterior} to $\phi' : {}^e M \to \Hom_R(S, L)$ we obtain:
\[
\varepsilon\big( \phi'( {}^e (M_{*S}) ))\big) \subseteq \varepsilon\big( \phi'( ({}^e M)_{*S} ))\big) \subseteq \varepsilon(\Hom_R(S, L)_{*S}).
\]
Finally, using co-persistence (Theorem \ref{thm:copers}) we obtain $\varepsilon(\Hom_R(S, L)_{*S}) \subseteq L_{*R}$ as desired.
\end{proof}

Combining, Proposition \ref{prop.InteriorTransformsUnderExtensions} and Corollary \ref{cor.PropEasyContainmentOfTauForFinite}, we obtain the following.

\begin{corollary}
\label{cor.TauTransformsUnderFiniteExtensions}
Suppose that $R$ is a $F$-finite reduced ring of characteristic $p > 0$ and $R \subseteq S$ is a finite extension with $S$ reduced.  Further suppose that $L$ is a finite $R$-module whose support agrees with $\Spec R$ and $M$ is a finite $S$-module whose support agrees with $\Spec S$.  Then:
\begin{equation}
\label{eq.InteriorDescriptionForFiniteMaps}
L_* = \sum_{e \geq 0} \sum_{\phi} \phi( {}^e(M_*) )
\end{equation}
where $\phi$ ranges over all elements of $\Hom_{R}( {}^e M, L)$.

In particular, if $L = R$ and $M = S$, then
\begin{equation}
\label{eq.TauDescriptionForFiniteMaps}
\taub(R) = \sum_{e \geq 0} \sum_{\phi} \phi( {}^e\taub(S) )
\end{equation}
where $\phi$ ranges over all elements of $\Hom_{R}( {}^e S, R)$.
\end{corollary}

\section{Test ideals, conductors, normalization and minimal primes}\label{sec:conductor}
In this section we explore test ideals of non-normal rings.  Earlier we showed how the \tint\ of a ring or module behaved modulo minimal prime ideals.  We now expand upon those ideas relating the test ideal (in other words $R_*$) with the test ideal of the normalization of $R$ in its total field of fractions.  As an application, we are able to prove that the big and finitistic test ideal agree in a non-normal ring if the normalization of that ring is strongly $F$-regular.  To do this, we apply the results of the previous section, Section \ref{sec.TransformationRulesForInteriors}.

Throughout this section, we fix a reduced Noetherian ring $R$ (not always of characteristic $p > 0$).  We let $\p_1, \dotsc, \p_n$ be its set of minimal primes, and we let $\bc=\bc(R)$ be the conductor of $R$.  For each $1\leq i \leq n$, let $\ia_i := \ann_R \p_i = \bigcap_{j \neq i} \p_j$.  Recall that $\taub(R) \subseteq \taufg(R)$ by definition.

\begin{remark}
If $R$ has equal characteristic (in other words, if it contains a field) then we can define tight closure of ideals \cite{HochsterHunekeTightClosureInEqualCharactersticZero} \cite[Appendix by M.~Hochster]{HunekeTightClosureBook}.  One can then define finitistic test ideals, $\tau_{\textnormal{fg}} = \bigcap_{I \subseteq R} (I^* : I)$ and the notion of weak $F$-regularity as usual.  In the first half of this section, we work in this equal characteristic setting when dealing with finitistic test ideals.  However, if the reader is unfamiliar with this generality, we invite him or her to restrict to the case where $R$ is of characteristic $p > 0$.
\end{remark}

\begin{prop}
\label{prop.SumOfAiInsideTauIfFRegular}
Assume $R$ is of equal characteristic and that each $R/\p_i$ is weakly $F$-regular (\resp assume $R$ is of characteristic $p > 0$ and each $R/\bp_i$ is strongly $F$-regular).  Then $\sum_{i=1}^n \ia_i \subseteq \taufg(R)$ (\resp $\subseteq \taub(R)$).
\end{prop}

\begin{proof}
We cover the weakly $F$-regular case first: By symmetry, it is enough to show that $\ia_1 \subseteq \taufg(R)$.  So let $c\in \ia_1$, let $I$ be an ideal of $R$, and $x\in I^* = \bigcap_{i=1}^n (I + \p_i)$.  Then $x\in I + \p_1$, so $cx \in I + \ia_1 \p_1 =I$.  Thus, $c\in \bigcap_I (I : I^*) = \taufg(R)$.

The strongly $F$-regular case is similar: Let $c\in \ia_1$,  let $M$ be an arbitrary $R$-module, and let $x\in 0^*_M$.  Since $0^*_M = \bigcap_{i=1}^n \p_i M$, we have $x \in \p_1 M$, so that $cx \in \ia_1 \p_1 M = 0$.  Thus, $c \in \bigcap_M \ann 0^*_M = \taub(R)$.
\end{proof}

\begin{prop}\label{pr:tfgcond}
If $R$ is of equal characteristic, then  we have $\taufg(R) \subseteq \bc$.
\end{prop}

\begin{proof}
Let $c \in \taufg(R)$.  Take any $x\in R^{\textnormal{N}}$ (where $R^{\textnormal{N}}$ is the integral closure of $R$ in its total ring of fractions).  Then $x = f/g$ for some $f, g \in R$ such that $g$ is a non-zerodivisor, and $f \in (g)^- = (g)^*$ (here $(g)^-$ denotes the integral closure of $(g)$).  But $c \cdot (g)^* \subseteq (g)$, whence $cf = gh$ for some $h \in R$, so that $cx = h \in R$.  Since $x$ was arbitrary, $c \in (R :_R R^{\textnormal{N}}) = \bc$.
\end{proof}

\begin{prop}\label{pr:condinann}
For any reduced Noetherian ring (of any characteristic), $\bc \subseteq \sum_{i=1}^n \ia_i$.
\end{prop}

\begin{proof}
Fix $c \in \bc$.
We have $R^{\textnormal{N}} \cong \prod_{i=1}^n (R/\p_i)^{\textnormal{N}}$, under which the natural inclusion $j: R \hookrightarrow R^{\textnormal{N}}$ sends $r \mapsto (\ov{r}, \ov{r}, \dotsc, \ov{r})$, where \raisebox{5pt}{$\ov{\;\,}$} denotes the image of the element in each normalized residue class ring.

Now, for each $1\leq i \leq n-1$, let $u_i := (\ov{1}, \ov{1}, \dotsc, \ov{1}, 0, \dotsc, 0) \in R^{\textnormal{N}}$, with $i$ entries $1$s and $(n-i)$ entries $0$s.  For each $i$, since $c u_i \in R$, there exists an element $b_i \in R$ such that \[
(\ov{c}, \ov{c}, \dotsc, \ov{c}, 0, \dotsc, 0) = c u_i = j(b_i) = (\ov{b_i}, \dotsc, \ov{b_i}).
\]
Let $a_i := c-b_i$ (that is, $c = a_i + b_i$).  The above equation means precisely that $a_i \in \bigcap_{j=1}^i \p_j$ and $b_i \in \bigcap_{j=i+1}^n \p_j$.

For each $2 \leq i \leq n-1$, observe that \[
a_{i-1} - a_i \in \ia_i.
\]
To see this, note that $c = a_{i-1} + b_{i-1} = a_i + b_i$  implies that $a_{i-1} - a_i = b_i - b_{i-1}$.  But $a_{i-1} - a_i \in \bigcap_{j=1}^{i-1} \p_j$ and $b_i - b_{i-1} \in \bigcap_{j=i+1}^n \p_j$, so since these are the same element, we have $a_{i-1} - a_i \in (\bigcap_{j=1}^{i-1} \p_j) \cap (\bigcap_{j=i+1}^n \p_j) = \ia_i$, as required.

Set $a_0 := c$ and $a_n :=  0$.  Then we also have $a_0-a_1 = b_1 \in \ia_1$ and $a_{n-1} - a_n = a_{n-1} \in \ia_n$, so that the latest displayed equation holds for $i=1, \dotsc, n$.  Altogether then, we have \[
c = \sum_{i=1}^n (a_{i-1} - a_i) \in \sum_{i=1}^n \ia_i,
\]
as was to be shown.
\end{proof}

Taken together, we draw the following conclusion which in characteristic $p > 0$ can also be viewed as a special case of \cite[Proposition 4.4]{SmithMultiplierTestIdeals}:
\begin{thm} \label{thm.ModMinimalPrimesIsFRegularThen}
Let $R$ be a reduced equicharacteristic Noetherian ring, with minimal primes $\p_1, \dotsc, \p_n$, and suppose each $R/\p_i$ is weakly $F$-regular.  Then \[
\taufg(R) = \bc = \sum_{i=1}^n \ann \p_i.
\]
If moreover, $R$ is of characteristic $p > 0$ and each $R/\p_i$ is \emph{strongly} $F$-regular, then $\taub(R)$ also coincides.
\end{thm}

\begin{remark}\label{rmk:modminprimes}
When $R$ is a Stanley-Reisner ring of positive characteristic and $\p_1, \dotsc, \p_n$ are its minimal primes, J. Vassilev \cite[Theorem 3.7]{Va-testquot} showed that $\taufg(R) = \sum_{i=1}^n \ann \p_i$, after which W. Traves gave a $D$-module proof of the same result in \cite[Theorem 5.8]{Tr-diffmon}.

Our theorem generalizes Vassilev's result, since for such a ring $R$, each $R/\p_i$ is a polynomial ring over a field, hence strongly $F$-regular.
\end{remark}

We note also the following, which may be already known to experts, but should be of independent interest:

\begin{prop}
Assume that each $R/\p_i$ is normal.  Then $\sum_{i=1}^n \ia_i = \bc$.
\end{prop}

\begin{proof}
By Proposition~\ref{pr:condinann}, we need only show that each of the $\ia_i$ is contained in $\bc$, and by symmetry we need only show it for $\ia_1$.  So let $c\in \ia_1$ and let $x\in R^{\textnormal{N}}$.  Then $x=f/g$ for some non-zerodivisor $g$ such that $f \in (g)^- = \bigcap_{i=1}^n \big( (g) + \p_i \big)$ by the assumption on the $R/\p_i$.  In particular, $f \in gR+\p_1$, so that $cf \in gR + c\p_1 = gR$, whence $cx \in R$.  Thus $c\in (R:_R R^{\textnormal{N}}) = \bc$.
\end{proof}

We will need the following characterization of the conductor ideal given by (iii) below.

\begin{prop}\cite{HunekeSwansonIntegralClosure}
\label{prop.CharsOfConductor}
Suppose that $R$ is a reduced excellent ring with normalization $R^{\textnormal{N}}$ in its total ring of fractions then the \emph{conductor of $R$ (in $R^{\textnormal{N}}$)}, denoted by $\bc$, is defined in any of the following equivalent ways:
\begin{itemize}
\item[(i)]  $\bc =   \Ann_R(R^{\textnormal{N}}/R)$.
\item[(ii)]  $\bc$ is the largest ideal of $R$ that is simultaneously an ideal of $R^{\textnormal{N}}$.
\item[(iii)]  $\bc = \sum_{\phi} \phi(R^{\textnormal{N}})$ where $\phi$ ranges over $\Hom_R(R^{\textnormal{N}}, R)$.
\end{itemize}
\end{prop}
\begin{proof}
We only describe the equivalence of (i) = (ii) with (iii) as the other equivalence is well known.  First suppose $x\in \bc$.  Then there is an $R$-linear map $m: R^{\textnormal{N}} \to R$ given by multiplication by $x$, so $m(1)=x$.

Conversely, let $g: R^{\textnormal{N}} \to R$ be an $R$-linear map, and let $x=g(1)$.  Fix $y\in R^{\textnormal{N}}$ and note that $y=a/s$ for some $a,s\in R$ with $s$ a non-zerodivisor.  We have $sg(y) = g(sy) = g(a) = ag(1) = ax$.  Thus, $xy = ax/s = g(y) \in R$.  Thus $x\in \bc$.
\end{proof}

We now transition to the characteristic $p > 0$ setting.
Consider the following lemma which is inspired by Proposition \ref{prop.CharsOfConductor}(iii).

\begin{lemma}
\label{lem.ImageFromNormalizationIsInConductor}
Suppose that $R$ is a $F$-finite reduced ring of characteristic $p > 0$ and $R \subseteq R^{\textnormal{N}}$ is its normalization with conductor ideal $\bc = \Ann_R(R^{\textnormal{N}}/R)$.  Fix some $e > 0$ and an $R$-linear map $\phi : {}^e(R^{\textnormal{N}}) \to R$.  Then $\Image(\phi) \subseteq \bc$.
\end{lemma}
\begin{proof}
It is sufficient to show that $\Image(\phi)$ is an $R^{\textnormal{N}}$-ideal as well since $\bc$ is the largest such ideal, so choose $x \in R^{\textnormal{N}}$, we need to show that $x \Image(\phi) \subseteq \Image(\phi)$.  Notice that by tensoring over $R$ with $K$, the total ring of fractions of $R$, we have the following commutative diagram:
\[
\xymatrix{
{}^e(R^{\textnormal{N}}) \ar@{^{(}->}[d] \ar[r]^{\phi} & R \ar@{^{(}->}[d] \\
{}^eK \ar[r]_{\overline{\phi}} & K
}
\]
where $\overline{\phi} : {}^eK \cong {}^e(R^{\textnormal{N}}) \tensor_{R} K \to K$ is induced by $\phi$.  Note also that $\overline{\phi}$ is $R^{\textnormal{N}}$-linear.  Now,
\[
x \phi\left( {}^e(R^{\textnormal{N}}) \right) = x \overline{\phi}\left( {}^e(R^{\textnormal{N}}) \right) = \overline{\phi}\left( x {}^e(R^{\textnormal{N}}) \right) \subseteq \overline{\phi}\left( {}^e(R^{\textnormal{N}}) \right) = \Image(\phi)
\]
which completes the proof.  Note that this also implies that any map $\phi : {}^e(R^{\textnormal{N}}) \to R \subseteq R^{\textnormal{N}}$ is also $R^{\textnormal{N}}$-linear.
\end{proof}

Now we apply our transformation rule for test ideals, Corollary \ref{cor.TauTransformsUnderFiniteExtensions}, to the inclusion $R \subseteq R^{\textnormal N}$.  This slightly generalizes \cite[Proposition 4.4]{SmithMultiplierTestIdeals} to the case of big test ideals and strongly $F$-regular rings.

\begin{thm}
\label{thm.TauForNormalizationFRegular}
Suppose that $R$ is an $F$-finite reduced ring of characteristic $p > 0$ and that $R \subseteq R^{\textnormal{N}}$ is its normalization with conductor $\bc$.  If $R^{\textnormal{N}}$ is strongly $F$-regular, then
\[
\bc = \taufg(R) = \taub(R).
\]
\end{thm}
\begin{proof}
First recall that $\tau_{\textnormal{fg}}(R) \subseteq \bc$, by Proposition~\ref{pr:tfgcond}.
Thus we now have the containments
\[
\taub(R) \subseteq \tau_{\textnormal{fg}}(R) \subseteq \bc.
\]
On the other hand by Corollary \ref{cor.TauTransformsUnderFiniteExtensions},
\[
\taub(R) = \sum_{e \geq 0} \sum_{\phi \in \Hom_R({}^e(R^{\textnormal{N}}), R)} \phi({}^e \taub(R^{\textnormal{N}})) = \sum_{e \geq 0} \sum_{\phi \in \Hom_R({}^e(R^{\textnormal{N}}), R)} \phi({}^e R^{\textnormal{N}})
\]
By considering only $e = 0$ we have that $\bc \subseteq \taub(R)$ by Proposition \ref{prop.CharsOfConductor}(iii) which completes the proof.
\end{proof}

\begin{remark}
This theorem generalizes D. McCulloch's result \cite[Lemma 5.3.3]{Mcc-thesis} that  $\taufg(R) = \bc(R)$ for a reduced binomial ring $R$ of positive characteristic, because the normalization of such a ring is strongly $F$-regular by Smith \cite[proof of the last Corollary]{Sm-binom}.
\end{remark}

\def\cprime{$'$} \def\cprime{$'$}
  \def\cfudot#1{\ifmmode\setbox7\hbox{$\accent"5E#1$}\else
  \setbox7\hbox{\accent"5E#1}\penalty 10000\relax\fi\raise 1\ht7
  \hbox{\raise.1ex\hbox to 1\wd7{\hss.\hss}}\penalty 10000 \hskip-1\wd7\penalty
  10000\box7}
\providecommand{\bysame}{\leavevmode\hbox to3em{\hrulefill}\thinspace}
\providecommand{\MR}{\relax\ifhmode\unskip\space\fi MR}
\providecommand{\MRhref}[2]{%
  \href{http://www.ams.org/mathscinet-getitem?mr=#1}{#2}
}
\providecommand{\href}[2]{#2}

\end{document}